\documentclass[11pt]{amsart}
\allowdisplaybreaks[4]
\usepackage{amssymb}
\usepackage{amsfonts}
\usepackage{graphicx}
\usepackage{epstopdf}
\usepackage{dcolumn}
\usepackage{amsmath}
\usepackage{enumerate}
\usepackage{latexsym,bm}
\usepackage{slashed}
\usepackage{float}
\usepackage{cite}

\usepackage[all,cmtip]{xy}
\usepackage{appendix}
\usepackage{xcolor}
\usepackage[colorlinks,
linkcolor=black,
citecolor=black,
urlcolor=black
]{hyperref}
\usepackage{tikz}
\usetikzlibrary{arrows.meta,shapes,chains}

\DeclareMathOperator{\Span}{Span}

\DeclareMathOperator{\sr}{sr}
\DeclareMathOperator{\cell}{cell}

\DeclareMathOperator{\Cay}{Cay}
\newcommand{\A}{\mathcal A}
\newcommand{\Z}{\mathbb Z}

\newcommand{\K}{\mathbb K}

\newcommand{\im}{\operatorname{im}}
\newcommand{\rank}{\operatorname{rank}}

\theoremstyle{plain}
\newtheorem{theorem}{Theorem}[section]
\newtheorem{lemma}[theorem]{Lemma}
\newtheorem{lem/def}[theorem]{Lemma/Definition}
\newtheorem{cor/def}[theorem]{Corollary/Definition}
\newtheorem{proposition}[theorem]{Proposition}
\newtheorem{corollary}[theorem]{Corollary}

\theoremstyle{definition}
\newtheorem{definition}[theorem]{Definition}
\newtheorem{example}[theorem]{Example}

\newtheorem{remark}[theorem]{Remark}
\numberwithin{equation}{section}

\allowdisplaybreaks

\begin{document}

\title{Path Homology of Circulant Digraphs}
\author{Xinxing Tang and Shing-Tung Yau}

  \address{
X. Tang: Beijing Institute of Mathematical Sciences and Applications, Beijing, China;
}
\email{tangxinxing@bimsa.cn}
  \address{
S.-T. Yau: Yau Mathematical Sciences Center, Tsinghua University, and Beijing Institute of Mathematical Sciences and Applications, Beijing, China;
}
\email{styau@tsinghua.edu.cn}

\begin{abstract}

We study GLMY path homology of circulant digraphs through step-word complexes and infinite cyclic lifts. We first prove chain retractions for
interval connection sets. For a fixed finite positive connection set $S$ containing $1$, the infinite lift yields a split exact sequence and all-degree stability over every field. We then encode the short additive relations of $S$ in a lattice lift and obtain exterior homology in arbitrarily high degrees for collision-free step sets. Under an exact Freiman admissibility condition, we classify all connection sets with three positive steps. For arbitrarily many steps, we compute rectangular additive configurations and arithmetic
progressions over every field, and doubling strings in characteristic zero, including their boundary ranks. 
\end{abstract}
\maketitle

\tableofcontents

\section{Introduction}

\subsection{Background and motivation}
The foundational GLMY paper \cite{GLMY:basic} introduces \emph{path complexes of digraphs} as a common generalization of simplicial complexes and cubical complexes (We recall their definitions in Section \ref{path}). Over the subsequent series, GLMY and their collaborators developed a toolkit that parallels standard algebraic topology: path cohomology theory \cite{GLMY:cohomology}, functoriality, homotopy invariance (for suitable notions of digraph homotopy) \cite{GLMY:homotopy}, and product formulas \cite{GMY:product}. In particular, K\"unneth-type theorems for Cartesian products and joins are proved at the chain-complex level, which is unusually strong compared to many graph-homology formalisms. 

In the last five years, there are rapidly increasing developments of GLMY theory. For example, Lin-Wang-Yau \cite{LWY} developed the discrete Morse theory on digraph. Di-Ivanov-Mukoseev-Zhang \cite{DIMZ} developed a theory of covering digraphs and applyed it to Cayley digraphs to build a bridge between GLMY path homology and group homology. Tang-Yau developed the structure theory \cite{TY1} for the path complex and the cellular complex \cite{TY2} for digraphs under a certain strongly regular condition. Ivanov-Pavutnitskiy \cite{IP} recasted GLMY-style path homology in a general simplicial framework. Fu-Ivanov \cite{FI} constructed explicit bases for path chain spaces for digraphs without multisquares. Li-Muranov-Wu-Yau \cite{LMWY} developed singular cubical/simplicial style homology theories inside categories of digraphs/quivers. The theory also interfaces with analysis on graphs via path Laplacians and Hodge-type decompositions \cite{GrigoryanHodge, GLYZ:eigenvalues}, making it relevant to directed network dynamics. 

\vskip 0.3cm

A \emph{circulant digraph} is a Cayley digraph of the cyclic group $\mathbb{Z}_n$: vertices are residues mod $n$, and directed edges are prescribed by a connection set $S\subset \mathbb{Z}_n$ via $a\to a+s$ for $s\in S$ \cite{CayleyMathWorld,BoeschTindell84,Alspach2002,CirculantThesis2016}. 
Circulants matter for several reasons.
Algebraically, they are the simplest nontrivial Cayley digraphs and thus a natural testing ground for phenomena in Cayley
digraph theory (connectivity, automorphism groups, isomorphism problems, enumeration).
Analytically and computationally, the translation symmetry of $\Z_n$ implies that adjacency and Laplacian-type operators
are diagonalizable by the discrete Fourier transform; see, for example, the classical references on circulant matrices
\cite{Davis1979,Gray2006} and surveys on circulant (di)graphs \cite{Morris06Survey}.
As a result, circulants are ubiquitous model families in spectral graph theory \cite{Biggs}, random walks/mixing \cite{ABTW}, coding theory \cite{Fossorier04, MBGcoding}, and algorithmic benchmarking \cite{BCH95, Hwang03, HRD2022}; they also appear naturally as regular motifs in network science.

\vskip 0.3cm
Circulant digraphs are an especially natural class where explicit GLMY computations are both informative and tractable. 
\begin{itemize}
    \item First, the $\mathbb{Z}_n$-action by translation induces an action on every $\Omega_m(G)$ commuting with the GLMY boundary. Over a splitting field of characteristic zero, the entire chain complex decomposes into Fourier blocks, and ranks of differentials reduce to smaller matrices. The arbitrary-characteristic results below instead use the infinite cyclic lift.
    \item Second, the ``allowed path'' condition becomes combinatorics of words in the connection set $S$, while the defining condition $\partial \Omega_m(G)\subseteq \Omega_{m-1}(G)$ becomes a system of local linear relations expressing cancellation of forbidden faces. This frequently yields closed-form bases for $\Omega_m(G)$ and makes higher-dimensional behavior accessible. 
    \item Third, circulants provide controlled test cases for structural theorems (homotopy invariance, K\"unneth formulas, Laplacians) and for persistent extensions, because natural filtrations (by weights, by enlarging $S$, by thresholds) preserve the underlying group symmetry \cite{ChowdhuryMemoli2017}. 
\end{itemize}

For these reasons, studying and computing path complexes/homology of circulant digraphs is not merely an exercise: it clarifies how GLMY topology detects directed cyclicity and higher-order directed motifs in a maximally symmetric setting, and it supplies explicit building blocks for both theory and applications.

\subsection{Main results}

Building on GLMY homotopy invariance \cite{GLMY:homotopy} and the covering
framework of Di--Ivanov--Mukoseev--Zhang \cite{DIMZ}, we obtain the
following results.

\begin{theorem}[Theorem~\ref{thm:S12d}]
For $1\le e\le d$ and $n>2d$, the inclusion
\[
i:\Omega_*(\vec C_n^{1,\ldots,e})
  \longrightarrow\Omega_*(\vec C_n^{1,\ldots,d})
\]
is a chain homotopy equivalence admitting a strict chain retraction. Over
every field, both complexes have one-dimensional homology in degrees 0,1 and zero homology in higher degrees.
\end{theorem}

\begin{theorem}[Theorems~\ref{cl:ses} and~\ref{cl:stability}]
\label{main:stability}
Let $S\subset\mathbb Z_{>0}$ be finite with $1\in S$, let $\K$ be any
field, and set $X_S=\Cay(\mathbb Z,S)$. If $n>2\max S$, then
\[
0\longrightarrow H_m^{\mathrm{path}}(X_S;\K)
\longrightarrow H_m^{\mathrm{path}}(\vec C_n^S;\K)
\longrightarrow H_{m-1}^{\mathrm{path}}(X_S;\K)
\longrightarrow0
\]
is split exact for every $m\ge0$. Moreover, a finite step-word complex
$C_*(S;\K)$, determined by the short additive relations among the steps,
satisfies
\[
H_m^{\mathrm{path}}(\vec C_n^S;\K)\cong H_m(C_*(S;\K)).
\]
\end{theorem}

\begin{theorem}[Theorem~\ref{st:exterior}]
Let $S\subset\mathbb Z_{>0}$ be finite with $1\in S$, and let $\K$ be
any field. If $(S+S)\cap S=\varnothing$ and every unordered pair of
steps has a distinct sum, then, as graded vector spaces,
\[
H_*^{\mathrm{path}}(\vec C_n^S;\K)\cong\Lambda^*(\K^S)
\qquad(n>2\max S).
\]
\end{theorem}

\begin{theorem}[Theorem~\ref{thm:complete-three-step}]
\label{main:triples}
Let $S=\{1,a,b\}$ with $1<a<b$, let $n$ be admissible for $S$ in
Definition~\ref{cl:freiman}, and let $\K$ be any field. Set
\[
r(a,b)=\begin{cases}
1,&(a,b)=(2,3)\text{ or }(2,4),\\
3,&a\ge3,\ b\notin\{a+1,2a-1,2a\},\\
2,&\text{otherwise}.
\end{cases}
\]
Then
\[
H_m^{\mathrm{path}}(\vec C_n^{1,a,b};\K)
\cong\Lambda^m(\K^{r(a,b)})
\qquad(m\ge0).
\]
\end{theorem}

The seven additive patterns underlying this formula are listed in
Table~\ref{tab:three-step-patterns}. In particular, the Betti sequences
for $\{1,2,4\}$, $\{1,2,5\}$, and $\{1,4,6\}$ are respectively
$(1,1)$, $(1,2,1)$, and $(1,3,3,1)$, followed by zeros.

\begin{theorem}[Theorem~\ref{thm:box-family}]
For a $d$-dimensional rectangular additive type satisfying the hypotheses
of Theorem~\ref{thm:box-family}, every field $\K$, and every admissible
modulus $n$,
\[
H_m^{\mathrm{path}}(\vec C_n^S;\K)\cong\Lambda^m(\K^{d+1}),
\qquad
H_m^{\mathrm{path}}(X_S;\K)\cong\Lambda^m(\K^d).
\]
In particular, every arithmetic progression
$S_{q,k}=\{1,1+q,\ldots,1+kq\}$ has stable Betti sequence
$(1,2,1,0,\ldots)$.
\end{theorem}

\begin{theorem}[Theorem~\ref{thm:doubling-strings}]
In characteristic zero, suppose unordered pair sums in $S$ are unique and
the only legal pair sums are doublings. If
$D=\{s\in S:2s\in S\}$ and $b=|S|-|D|$, then
\[
\dim_\K H_m^{\mathrm{path}}(\vec C_n^S;\K)=\binom bm
\qquad(n>2\max S).
\]
\end{theorem}

\subsection{Organization}
In Section~\ref{path}, we recall GLMY path homology and give the motivating
example $\vec C_5^{1,2}$. In Section~\ref{sec:circulant-basics}, we derive the local cancellation equations and deal with two-step and interval connection
sets. In Section~\ref{st:section}, we construct the stable step-word complex and
prove the infinite-cyclic-lift exact sequence. In Section~\ref{sec:three-step-classification},
we introduce the short-relation lift and prove the complete three-step
classification. In Section~\ref{sec:box-families}, we consider the rectangular additive
types and doubling strings with arbitrarily many steps. In Appendix \ref{sec:circulantgraph}, we
discuss marked lattice quotients and symmetric circulants, while in
Appendix \ref{app:cohomology}, we briefly discuss the cohomology result.

\vskip 0.2cm
\noindent\textbf{Acknowledgments.} X. Tang thanks Jiedong Jiang for pointing out an error in a computation in the first version of this paper.
X. Tang is supported by the Youth Project No. 12501079 of NSFC and the start-up fund in BIMSA. Part of the paper is completed by X. Tang during her visit to SIMIS, and she extends her gratitude to the host for their warm reception and thoughtful arrangements.

\section{The path complex of digraphs}\label{path}

\subsection{The definition of path complex}\label{sub:pathcpx}

Let $V$ be a set. For $n\geq0$, an elementary $n$-path in $V$ is a sequence $p=(v_0,v_1,\dots,v_n)$, also written as $e_{v_0v_1\ldots v_n}$.

Let $\K$ be a field, and let $\Lambda_n(V;\K)$ be the free $\K$-vector space on the elementary $n$-paths. One can define a $\K$-linear map $\partial:\Lambda_n(V;\K)\rightarrow \Lambda_{n-1}(V;\K)$ as follows:
\begin{equation}\label{eq:boundary}
\partial e_{v_0\dots v_n}=\sum_{j=0}^n (-1)^j\,e_{v_0\dots\widehat v_j\dots v_n}.
\end{equation}
It is easy to check that $(\Lambda_*(G;\K),\partial)$ forms a complex.

The set of regular $n$-paths is defined as a quotient space $R_n(V;\K)$:
$$R_n(V;\K)=\Lambda_n(V;\K)/I_n(V;\K),$$
where $I_n(V;\K)$ is the $\K$-subspace generated by the irregular path $e_{v_0v_1\ldots v_n}$, where $v_{k-1}=v_k$ for some $k=1,\ldots,n$. 
\begin{proposition}
$(R_*(V;\K),\partial)$ forms a quotient chain complex with the induced boundary operator.
\end{proposition}

\begin{definition}
Let $G=(V,E)$ be a digraph without self-loops, where $V$ is the set of vertices, and $E$ is the set of directed edges.

The space of allowed $n$-paths in $G$, denoted by $A_n(G;\K)$, is defined to be
$$A_n(G;\K)=\Span_{\K}\{e_{v_0v_1\cdots v_n}\in R_n(V;\K)~\big|~v_kv_{k+1}\in E, ~k=0,1,\ldots,n-1.\}.$$
\end{definition}

Since the boundary of an elementary $n$-path is the alternating sum of $(n-1)$-faces.
After deleting an interior vertex, two consecutive edges are merged into one, so a face may fail to be allowed. That is, the boundary operator $\partial$ does not preserve the space $A_*(G;\K)$. GLMY defines the $\partial$-invariant subspaces $\Omega_n(G;\K)\subseteq A_n(G;\K)$ recursively by
\begin{align*}
&\Omega_0(G;\K)=A_0(G;\K),\qquad \Omega_1(G;\K)=A_1(G;\K),\\
&\Omega_n(G;\K)=\{u\in A_n(G;\K):\partial u\in A_{n-1}(G;\K)\}\qquad (n\ge 2),
\end{align*}
Here $\partial u$ is computed in the regular quotient $R_{n-1}(V;\K)$: irregular faces are zero, and the coefficients of all remaining non-allowed faces must cancel. By construction, $\partial(\Omega_n(G;\K))\subseteq \Omega_{n-1}(G;\K)$, so $(\Omega_*(G;\K),\partial)$
is a chain complex, known as the path complex of the digraph $G$, or GLMY complex of $G$. Its homology is called the path homology or GLMY homology of the digraph $G$:
\[
H_n^{\mathrm{path}}(G;\K)=H_n(\Omega_*(G;\K),\partial).
\]

\subsection{An interesting example}\label{sub:example1}

In this subsection, we will illustrate an interesting example to help the readers understand the path homology of digraphs.

Let $G$ be the following digraph, as Figure \ref{fig:C512}. We will denote it by $\vec{C}_5^{1,2}$ in the next section.
\begin{figure}[h]
\centering
\begin{tikzpicture}[
scale=0.8, 
    >=Stealth,
    node/.style={inner sep=0.6pt} 
]

\node[node] (0) at (90:2)   {0};   
\node[node] (1) at (18:2)   {1};   
\node[node] (2) at (-54:2)  {2};   
\node[node] (3) at (-126:2) {3};  
\node[node] (4) at (162:2)  {4};

\draw[->,line width=0.8pt] (0) -- (1);
\draw[->,line width=0.8pt] (0) -- (2);
\draw[->,line width=0.8pt] (1) -- (2);
\draw[->,line width=0.8pt] (1) -- (3);
\draw[->,line width=0.8pt] (2) -- (3);
\draw[->,line width=0.8pt] (2) -- (4);
\draw[->,line width=0.8pt] (3) -- (4);
\draw[->,line width=0.8pt] (3) -- (0);
\draw[->,line width=0.8pt] (4) -- (0);
\draw[->,line width=0.8pt] (4) -- (1);

\end{tikzpicture}
\caption{circulant digraph $G=\vec{C}_5^{1,2}$} 
\label{fig:C512} 
\end{figure}
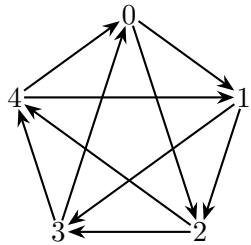

We have the following simple observations:
\begin{itemize}
  \item $e_{012}\in\Omega_2(G)$, $e_{013}\notin\Omega_2(G)$, $e_{013}-e_{023}\in\Omega_2(G)$;
  \item $e_{0123}\in\Omega_3(G)$, $e_{0124}-e_{0134}+e_{0234}\in\Omega_3(G)$;
  \item $e_{01234}\in\Omega_4(G)$, $e_{02340}-e_{01340}+e_{01240}-e_{01230}\in\Omega_4(G)$;
  \item For each vertex $a\in\mathbb{Z}_5$, the following elementary path is $\partial$-invariant:
  $$e_{a,(a+1),(a+2),\ldots,(a+n)}\in\Omega_{n}(G),\qquad a+j\pmod 5.$$
\end{itemize}

Thus, the path complex $\Omega_*(G;\mathbb{K})$ is infinite dimensional, where for each $n\in\Z_{\geq0}$, $\Omega_n(G;\mathbb{K})\neq0$. More explicitly, we have the following result.

\begin{lemma}\label{basis} For the circulant digraph $G=\vec{C}_5^{1,2}$, we have
\[\dim_{\mathbb{K}}\Omega_n(G;\mathbb{K})=10,\quad \forall~n\geq1.\]
Each vertex $a\in\Z_5$ gives two generators:
\begin{align*}
\alpha_a^{(n)}=~&e_{a,(a+1),(a+2),\ldots,(a+n)},\\
\beta_a^{(n)}=~&\sum_{j=1}^n(-1)^{n-j}e_{a,a+1,\ldots,a+j-1,a+j+1,a+j+2,\ldots,a+n+1}.
\end{align*}
Here all the vertices are understood modulo $5$.
\end{lemma}

\begin{proof} (1) It is easy to check that 
\[\alpha_a^{(n)},\beta_a^{(n)}\in\Omega_n(G;\mathbb{K}).\]
And $\alpha_a^{(n)}$ is the unique chain containing the elementary path $e_{a,a+1,\dots,a+n}$, and $\beta_a^{(n)}$ is the unique chain containing the elementary path $e_{a,a+2,a+3,\dots,a+n+1}$. Thus, all of the ten paths are linearly independent. Then we have $\dim_{\mathbb{K}}\Omega_n(G;\mathbb{K})\geq10$. 

On the other hand, for each vertex $a\in\Z_5$, let $A_n(a)\subset A_n$ be the span of allowed $n$-paths starting at $a$.
Then $A_n=\bigoplus_{a\in\Z_5} A_n(a)$ and $\Omega_n=\bigoplus_{a\in\Z_5} \Omega_n(a)$ where
\[\Omega_n(a):=\Omega_n\cap A_n(a).\]
It follows from the following lemma \ref{lem:fixed-start-2dof} that $\dim \Omega_n(a)\le 2$.

Thus, $\dim_{\mathbb{K}}\Omega_n(G;\mathbb{K})\leq10$, $n\geq1$. Then $\left\{\alpha_a^{(n)},\beta_a^{(n)}\right\}_{a\in\Z_5}$ is a basis of $\Omega_{n}(G;\mathbb{K})$.
\end{proof}

\begin{lemma}
\label{lem:fixed-start-2dof}
Fix $n\ge 2$ and $a\in\Z_5$.
Encode an allowed $n$-path starting at $a$ by its step word
\[w=(\varepsilon_1,\dots,\varepsilon_n)\in\{1,2\}^n,\]
which means that the path is
\[e_w=e_{a, a+\varepsilon_1, a+\varepsilon_1+\varepsilon_2, \cdots,~ a+\varepsilon_1+\cdots+\varepsilon_n}\quad(\mathrm{mod}\ 5).\]
Write any chain in $A_n(a)$ as $P=\sum_w c_w\,e_w$.
Then the condition $P\in\Omega_n(a)$ forces:
\begin{enumerate}
\item If $w$ contains the substring $22$ anywhere, then $c_w=0$.
\item For each position $j\in\{1,\dots,n-1\}$ and for each fixed context (letters outside positions $j,j+1$),
if $w=\cdots 12\cdots$ and $w'=\cdots 21\cdots$ differ only by swapping $12\leftrightarrow 21$ at $j,j+1$, then
\[c_w+c_{w'}=0.\]
\end{enumerate}
Consequently, $c_w=0$ for every word containing at least two occurrences of $2$.
Moreover, among words containing exactly one $2$, the coefficients are determined up to a single scalar by the sign rule
obtained from repeated swaps.
Therefore $\dim\Omega_n(a)\le 2$.
\end{lemma}

\begin{proof}
By definition of the boundary operator $\partial$, 
if the consecutive steps are $(\varepsilon_j,\varepsilon_{j+1})$, then the merged step is
$\varepsilon_j+\varepsilon_{j+1}\in\{2,3,4\}$.
In this circulant digraph $G$, only steps $1$ and $2$ are edges, so a merged step $3$ or $4$ produces a forbidden face.
The definition $x\in\Omega_n(a)$ is equivalent to requiring that every forbidden face has coefficient $0$ in $\partial x$.
We now read off the resulting linear equations.

\smallskip
\noindent\textbf{(i) Words containing $22$ vanish.}
Suppose $w$ contains $22$ at positions $j,j+1$.
Deleting the interior vertex between these two steps produces a face with a merged step $4$ at that location.
No other allowed word produces the same face with a step $4$ at the same location (because the only way to get $4$ is $2+2$).
Hence the coefficient of that forbidden face in $\partial x$ is exactly $\pm c_w$.
Requiring it to be $0$ forces $c_w=0$.

\smallskip
\noindent\textbf{(ii) Swapping $12$ and $21$ flips sign.}
Fix $j\in\{1,\dots,n-1\}$ and fix all letters of the word outside positions $j,j+1$.
There are exactly two allowed local patterns whose sum is $3$, namely $12$ and $21$.
Deleting the interior vertex between these two steps produces the same forbidden face with merged step $3$.
Moreover, in \eqref{eq:boundary} the sign $(-1)^j$ attached to deleting that vertex is the same for both words.
Therefore the coefficient of this forbidden face in $\partial x$ is $\pm(c_w+c_{w'})$, and requiring it to vanish gives
$c_w+c_{w'}=0$.

\smallskip
\noindent\textbf{(iii) Eliminating words with two or more $2$'s.}
Let $w$ contain at least two $2$'s.
If it contains $22$, then $c_w=0$ by (i).
Otherwise, between any two $2$'s there is at least one $1$.
Using (ii), we may swap an adjacent $21$ to $12$ (picking up a sign) and thereby move a $2$ one position to the right.
Repeating, we can move a $2$ across a string of $1$'s until it becomes adjacent to the next $2$, producing a word that contains $22$.
The coefficient of that adjacent word is $\pm c_w$ after the sequence of swaps.
By (i) it must be $0$, hence $c_w=0$.

\smallskip

In summary, the only potentially nonzero coefficients are:
\begin{itemize}
    \item[(a)] the all-$1$ word $1\cdots 1$ (one scalar), and
    \item[(b)] words with exactly one $2$. Any two such words differ by moving the unique $2$ through adjacent swaps; by (ii) each swap flips the sign, so all these coefficients
are determined by a single scalar.
\end{itemize}
Thus $\Omega_n(a)$ is spanned by two elements, hence $\dim\Omega_n(a)\le 2$.
\end{proof}

\begin{proposition}\label{prop:example-C512} The path homologies of the circulant digraph $G$ are given by
\[H_n^{\mathrm{path}}(G;\mathbb{K})\cong\begin{cases}
\mathbb{K} &\quad n=0,1\\
0 &\quad n\geq2
\end{cases}~.\]
\end{proposition}

\begin{proof}

 Let $s_n:=(-1)^n$. Using the basis in Lemma \ref{basis}, we have the following explicit formulas
(for $n\ge 3$, and again all indices mod $5$):
\begin{equation}\label{eqrec}
\begin{aligned}
 \partial \alpha_a^{(n)}&=s_n\alpha_a^{(n-1)}+\beta_{a+1}^{(n-1)}- s_n\beta_a^{(n-1)},\\
 \partial \beta_a^{(n)}&=s_n\alpha_a^{(n-1)}-s_n\alpha_{a+2}^{(n-1)}- s_n\beta_a^{(n-1)}+\beta_{a+1}^{(n-1)}.
\end{aligned}
\end{equation}
\vskip 0.2cm

Consider the digraph map (rotation or shift automorphism)
$$\tau:G\rightarrow G,\quad \tau(a)=a+1\pmod5.$$
Immediately we have
\begin{itemize}
\setlength\itemsep{0pt}
\item $\tau$ is an isomoprhism of $G$;
\item $\Omega_n(G;\mathbb{K})$ is $\tau$-invariant; in particular,
\[\tau\alpha_a^{(n)}=\alpha_{a+1}^{(n)},\qquad \tau\beta_a^{(n)}=\beta_{a+1}^{(n)}.\]
\item The boundary operator $\partial$ commutes with $\tau$.
\end{itemize}

Then we can compute the rank of $\partial_n$ by decomposing into $\tau$-eigenspaces. Over a field containing a primitive fifth root of unity (e.g. over $\mathbb{C}$), $\tau$ is diagonalizable with eigenvalues $\lambda$ satisfying $\lambda^5=1$.

For such a $\lambda$, define the Fourier modes
\[
\alpha^{(n)}_\lambda:=\sum_{a\in\Z_5}\lambda^{-a}\,\alpha^{(n)}_a,\qquad
\beta^{(n)}_\lambda:=\sum_{a\in\Z_5}\lambda^{-a}\,\beta^{(n)}_a.
\]
A change of variables $b=a+1$ shows
\[
\tau \alpha^{(n)}_\lambda=\sum_a \lambda^{-a}\alpha^{(n)}_{a+1}
=\sum_b \lambda^{-(b-1)}\alpha^{(n)}_b
=\lambda\sum_b \lambda^{-b}\alpha^{(n)}_b
=\lambda \alpha^{(n)}_\lambda,
\]
and similarly $\tau \beta^{(n)}_\lambda=\lambda \beta^{(n)}_\lambda$. Since the $\lambda$-eigenspace inside $\Omega_n$ has
dimension $2$ (one $\alpha$-type and one $\beta$-type generator), the pair $(\alpha^{(n)}_\lambda,\beta^{(n)}_\lambda)$ spans the $\lambda$-eigenspace.

Now applying the explicit boundary recursions \eqref{eqrec}, we can compute
\begin{align*}
\partial_n \alpha^{(n)}_\lambda
&=\sum_a \lambda^{-a}\partial_n \alpha^{(n)}_a\\
&=\sum_a \lambda^{-a}\Bigl(s_n\alpha^{(n-1)}_a+\alpha^{(n-1)}_{a+1}-s_n\beta^{(n-1)}_a\Bigr)\\
&=s_n\alpha^{(n-1)}_\lambda+\sum_a \lambda^{-a}\alpha^{(n-1)}_{a+1}-s_n\beta^{(n-1)}_\lambda\\
&=s_n\alpha^{(n-1)}_\lambda+\lambda \alpha^{(n-1)}_\lambda-s_n\beta^{(n-1)}_\lambda\\
&=(s_n+\lambda)\alpha^{(n-1)}_\lambda+(-s_n)\beta^{(n-1)}_\lambda.
\end{align*}
Similarly, we have
\begin{align*}
\partial_n \beta^{(n)}_\lambda
&=\sum_a \lambda^{-a}\partial_n \beta^{(n)}_a\\
&=\sum_a \lambda^{-a}\Bigl(s_n\alpha^{(n-1)}_a-s_n\alpha^{(n-1)}_{a+2}-s_n\beta^{(n-1)}_a+\beta^{(n-1)}_{a+1}\Bigr)\\
&=s_n\alpha^{(n-1)}_\lambda-s_n\sum_a \lambda^{-a}\alpha^{(n-1)}_{a+2}-s_n\beta^{(n-1)}_\lambda+\sum_a \lambda^{-a}\beta^{(n-1)}_{a+1}\\
&=s_n\alpha^{(n-1)}_\lambda-s_n\lambda^2\alpha^{(n-1)}_\lambda-s_n\beta^{(n-1)}_\lambda+\lambda\beta^{(n-1)}_\lambda\\
&=s_n(1-\lambda^2)\alpha^{(n-1)}_\lambda+(\lambda-s_n)\beta^{(n-1)}_\lambda.
\end{align*}
Thus, with respect to the eigenbases $(\alpha^{(n)}_\lambda,\beta^{(n)}_\lambda)$ and $(\alpha^{(n-1)}_\lambda,\beta^{(n-1)}_\lambda)$, the matrix of $\partial_n$ is given by
\[M_{s_n}(\lambda)=\begin{pmatrix}
s_n+\lambda & -s_n\\
s_n(1-\lambda^2) & \lambda-s_n.
\end{pmatrix}\]

A direct determinant computation gives (using $s_n=(-1)^n$)
\[\det M_{s_n}(\lambda)=(s_n+\lambda)(\lambda-s_n)-(-s_n)\cdot s_n(1-\lambda^2)=(\lambda^2-1)+(1-\lambda^2)=0.\]
Thus each block has rank at most 1. The entry $-s_n\neq0$, thus $M_{s_n}(\lambda)$ is a rank one 2-by-2 matrix. Since there are five eigenvalues, the total rank is $5$. Thus, for $n\geq3$, we have
\begin{align*}
 &\dim\im(\partial_n)=\rank(\partial_n)=5,\\
 &\dim\ker(\partial_n)=\dim\Omega_n-\rank{\partial_n}=10-5=5.
\end{align*}
Thus $H_n^{\mathrm{path}}(G;\K)=0$, for $n\geq3$. It is easy to compute that 
\[H_2^{\mathrm{path}}(G;\K)=0,\quad H_1^{\mathrm{path}}(G;\K)=\K[e_{01}+e_{12}+e_{23}+e_{34}+e_{40}],\quad H_0^{\mathrm{path}}(G;\K)\cong\K.\qedhere\]
\end{proof}

\begin{remark}
For simplicity, we introduce the stronger condition in our paper \cite{TY1,TY2}, and define the resulting path complex $H^{\sr}(G;\K)$ and cellular complex $H^{\cell}(G;\K)$ of the digraph $G$. In particular, we compute the two homologies (Example 4.3 in \cite{TY2}) of this example, we also get
\[
H_n^{\sr}(G;\K)\cong H_n^{\cell}(G;\K)\cong H_n^{\mathrm{path}}(G;\K)\cong
\begin{cases}
    \K & n=0,1\\
    0 & n\geq2
\end{cases}~.
\]
\end{remark}

\begin{remark}
In this example, we use the Fourier calculation as a motivating computation. The
field-uniform form of the same stability phenomenon, including the case $\operatorname{char}\mathbb K\mid n$, is proved later by the infinite cyclic lift in Section \ref{st:section}. It provides a much shorter computation. 
\end{remark}

\section{Circulant digraphs and the path complex}\label{sec:circulant-basics}
In this section, we study the path complex of circulant digraphs and compute the path homology in some special cases.
\subsection{Circulant digraphs}\label{sub:circulant}

\begin{definition}
Let $H$ be a (finite) group and $S\subset H$ with the unit $e\notin S$. The \emph{Cayley digraph} $G=\Cay(H,S)$ has vertex set $V=H$ and a directed edge $a\rightarrow as$ for each $a\in H$, $s\in S$.   
\end{definition}

\begin{remark}
If $S=S^{-1}$, we usually regard $G$ as an undirected $|S|$-regular Cayley graph.
\end{remark}

\begin{proposition}
There are many basic properties for the Cayley digraphs:
\begin{itemize}
    \item \textbf{vertex transitivity:} Left mulitiplication $L_h:a\mapsto ha$ is an automorphism of $\Cay(H,S)$ for every $h\in H$. In particular, Cayley digraphs are vertex-transitive.
    \item \textbf{connectivity:} For a finite group $H$, $\Cay(H,S)$ is (strongly) connected if and only if $\langle S\rangle=H$.
\end{itemize}
\end{proposition}

Fix an integer $n\geq 3$. Let $\mathbb Z_n=\mathbb Z/n\mathbb Z$ and let $S\subset\Z_n$ be a subset, $0\notin S$.

The \emph{circulant directed graph} $\vec{C}_n^S$ is the Cayley digraph 
$\Cay(\Z_n,S)$. That is, it has vertex set $V=\mathbb Z_n$ and an arrow
\[
a\longrightarrow a+s\quad\text{for each }a\in\mathbb Z_n,\ s\in S.
\]
This is the Cayley digraph of the cyclic group $\mathbb Z_n$ with connection set $S$. The example in the previous section is exactly $\vec{C}_5^{1,2}$. The circulant digraphs are important from several viewpoints: symmetry and classification, spectral/Fourier analysis, extremal and structural graph theory, algorithmic routing and fault-tolerance, network design, and interactions with topology (path homology) and operator algebras (graph $C^\ast$-algebras). 

\begin{remark}
One can also realize the circulant digraph from the matrix viewpoint. That is, the adjacency matrix of $\vec{C}_n^S$ is a circulant matrix. For $\vec{C}_5^{1,2}$, its adjacency matrix is given by
\[A=\begin{pmatrix}
    0 & 1 & 1 & 0 & 0\\
    0 & 0 & 1 & 1 & 0\\
    0 & 0 & 0 & 1 & 1\\
    1 & 0 & 0 & 0 & 1\\
    1 & 1 & 0 & 0 & 0
\end{pmatrix}.\]
Such a circulant matrix is diagonal by the discrete Fourier transform. This is the engine behind explicit spectral formulas and fast algorithms. 
\end{remark}

Now let us focus on the structure of path complex of the circulant digraph $G=\vec{C}_n^S$. 

First, we write the elementary $m$-path $e_{v_0,v_1\ldots,v_m}$ as
\begin{equation}\label{elementarypath2}
(v_0;~s_1,\ldots,s_m),\qquad s_j=v_j-v_{j-1},~~~j=1,2\ldots,m.
\end{equation}
It is an allowed elementary $m$-path if and only if all $s_j\in S$.

Let $p_m$ be the elementary $m$-path in $G$. 
Then
\[A_m(G)=\Span_{\K}\{p_m\}\cong \K[\Z_n]\otimes (\K^S)^{\otimes m}.\]

If $j$ is a starting or ending vertex of an allowed elementary path, the face is automatically allowed.
If $1\le j\le m-1$ is internal, deleting $v_j$ merges two consecutive steps:
\[
(s_j,s_{j+1}) \mapsto s_j+s_{j+1} \pmod n.
\]

An internal face is allowed when $s_j+s_{j+1}\in S$ modulo $n$; or when the merged step is zero it is irregular and represents zero in the regular quotient. 

\subsection{Low-dimensional path chains \texorpdfstring{$\Omega_2$, $\Omega_3$}{Omega2, Omega3} in circulant digraphs}\label{sub:lowdim}

Define a set
\[F:=(S+S)\setminus(S\cup\{0\})\subset\mathbb Z_n.\]

And we call $F$ the set of illegal merged steps. For $r\in\Z_n$, define the decomposition number
\[N(r):=\#\{(s,t)\in S\times S:\ s+t\equiv r\pmod n\}.\]

\subsubsection{\texorpdfstring{$\Omega_2(\vec{C}_n^S)$}{Omega2}}
\begin{proposition}
\label{prop:Omega2}
Fix a starting vertex $a\in\Z_n$. An allowed $2$-chain
$$
u=\sum_{s,t\in S} c_{s,t}\,e_{a,a+s,a+s+t}\in A_2
$$
belongs to $\Omega_2$ if and only if for every illegal $r\in F$, we have
\[
\sum_{s+t=r} c_{s,t}=0.
\]
Equivalently, for each illegal $r\in F$, the coefficients on the $N(r)$ paths with merged step $r$
have total sum $0$, hence contribute a subspace of dimension $N(r)-1$ (if $N(r)\ge 1$).
\end{proposition}

\begin{proof}
The proof follows from definition.   
\end{proof}

\begin{remark}
Here all sums are taken in $\mathbb Z_n$, and zero sums are omitted because they give irregular faces. In the positive no-wrap range $n>2\max S$, zero sums do not occur, so the condition reduces to the
integer-step version used below.
\end{remark}

\begin{definition}\label{def:Wcycle}
For $|S|\geq2$, $\{s,t\}\subset S$ and $a\in\mathbb Z_n$, we define the commutator square as follows
\[
\operatorname{comm}_a(s,t):=e_{a,a+s,a+s+t}\;-\;e_{a,a+t,a+s+t}.\footnote{For $|S|=1$, this commutator square is automatically 0.}
\]
Let its translation-average be
\[
W^{s,t}:=\sum_{a\in\mathbb Z_n}\operatorname{comm}_a(s,t)\in A_2(\vec{C}_n^S).
\]
\end{definition}

\begin{lemma}\label{lem:Wcycle}
For $|S|\ge 2$ and any distinct $s,t\in S$, we have 
\[W^{s,t}=-W^{t,s}\in\Omega_2(\vec{C}_n^S)\quad\text{ and }\quad \partial W^{s,t}=0.\]
\end{lemma}

\begin{proof}
The middle-deletion faces in $\partial(\operatorname{comm}_a(s,t))$ cancel, so each
$\operatorname{comm}_a(s,t)$ lies in $\Omega_2(\vec{C}_n^S)$. A direct expansion gives
\[
\partial(\operatorname{comm}_a(s,t))
=e_{a,a+s}-e_{a,a+t}+e_{a+s,a+s+t}-e_{a+t,a+s+t}.
\]
Summing over $a$ yields telescopic cancellation in pairs, hence $\partial W^{s,t}=0$.
\end{proof}

Whether $W^{s,t}$ survives in $H_2^{\mathrm{path}}(\vec{C}_n^S)$ depends on whether the graph has enough ``diagonals'' to fill it by $3$-chains.

\subsubsection{\texorpdfstring{$\Omega_3(\vec{C}_n^S)$}{Omega3}}

\begin{proposition}\label{prop:P}
Assume $|S|\geq2$ and  $s,t,r\in S$ satisfy $r\equiv s+t\pmod n$. Define the translation-average $3$-chain
\[
P^{s,t}:=\sum_{a\in\mathbb Z_n}e_{a,a+s,a+r,a+r+s}.
\]
Then $P^{s,t}\in \Omega_3(\vec{C}_n^S)$ and
\[
\partial P^{s,t}= W^{s,r}-W^{s,t}.
\]
Consequently, $[W^{s,r}]=[W^{s,t}]$ in $H_2^{\mathrm{path}}(\vec{C}_n^S)$. 
\end{proposition}

\begin{proof}
It is easy to compute
\begin{align*}
\partial e_{a,a+s,a+r,a+r+s}=~&e_{a+s,a+r,a+r+s}-e_{a,a+r,a+r+s}+e_{a,a+s,a+r+s}-e_{a,a+s,a+r}\\
=~&e_{a+s,a+r,a+r+s}-e_{a,a+s,a+r}+\operatorname{comm}_a(s,r)\\
=~&(a+s;~t,s)-(a;s,t)+\operatorname{comm}_a(s,r)\quad (\text{by notation} \eqref{elementarypath2})
\end{align*}
Thus, summing over $a\in\Z_n$, we arrive at
\[\partial P^{s,t}= W^{s,r}-W^{s,t}.\qedhere\]
\end{proof}

\begin{remark}
A relation $r=s+t\in S$ identifies $[W^{s,r}]$ and $[W^{s,t}]$. For $s=t$, it gives $[W^{s,2s}]=0$ because $W^{s,s}=0$. In particular it explains the vanishing of the square class for $S=\{1,2\}$. The general effect of legal merges and pair-sum collisions is described by the stable complex in Section \ref{st:section}.
\end{remark}

\subsection{Special cases}\label{sub:special}

\subsubsection{The case \texorpdfstring{$S=\{1\}$}{S1}}

In this case, we have $\langle S\rangle=\Z_n$. Thus, the digraph $\vec{C}_n^1$ is strongly connected. In particular, it is easy to obtain
\begin{itemize}
    \item[(1)] For $n=2$, the corresponding digraph is contractible, then 
    \[H_m(\vec{C}_2^1;\K)=\begin{cases}
        \K & m=0\\
        0 & m\ge 1.
    \end{cases}\] 
    \item[(2)] For $n>2$, the corresponding digraph is given by a monotone cycle. Then
    \[H_m(\vec{C}_n^1;\K)=\begin{cases}
        \K & m=0,1\\
        0 & m\ge 2.
    \end{cases}\]
\end{itemize}

\subsubsection{The case \texorpdfstring{$n\ge 5$, $S=\{1,s\}$, $1<s<n/2$}{S1s}.}
In this case, consider an allowed elementary $2$-path $e_{a,a+l,a+l+m}$ with $l,m\in\{1,s\}$.
The internal face deletion merges $l+m$.
Since $1<s<n/2$, the sums $2$, $s+1$ and $2s$ are distinct integers in $\{2,\dots,n-1\}$.

\begin{lemma} For $n\ge 5$, $S=\{1,s\}$, $1<s<n/2$, we have the following complete description for $\Omega_*(\vec{C}_n^S)$.
\begin{itemize}
\item[(1)] If $s=2$, then for $m\ge1$, $\Omega_m$ has a $2n$-element basis $\{\alpha_a^{(m)},\beta_a^{(m)}\}_{a\in\Z_n}$, where
\begin{align*}
\alpha_a^{(m)}=~&e_{a,(a+1),(a+2),\ldots,(a+m)},\\
\beta_a^{(m)}=~&\sum_{j=1}^m(-1)^{m-j}e_{a,a+1,\ldots,a+j-1,a+j+1,a+j+2,\ldots,a+m+1}.
\end{align*}
\item[(2)] If $s\ne 2$, then $\Omega_2$ has an $n$-element basis $\{\gamma_a\}_{a\in\Z_n}$, where
\[
 \gamma_a=e_{a,a+1,a+1+s}-e_{a,a+s,a+1+s}=\operatorname{comm}_a(1,s),
\]
and $\Omega_m(\vec{C}_n^S)=0$, $m\geq3$.
\end{itemize}
\end{lemma}

\begin{proof}
(1) The proof for the case $S=\{1,2\}$ is the same as the one of Lemma \ref{basis}. 

\vskip 0.2cm
(2) We write the allowed elementary $2$-path $e_{a,a+s_1,a+s_1+s_2}$ as $(a;s_1,s_2)$, and call it the pattern $(s_1,s_2)$. For $s\neq2$, 
\begin{itemize}
    \item the boundary of the pattern $(1,1)$ produces illegal path $e_{a,a+2}$;
    \item the boundary of the pattern $(s,s)$ produces illegal path $e_{a,a+2s}$;
    \item the mixed patterns $(1,s)$ and $(s,1)$ both produce the same illegal path $e_{a,a+s+1}$, but the difference $\gamma_a$ cancels it, so $\gamma_a\in\Omega_2(\vec{C}_n^S)$.
\end{itemize}
For $m\geq3$, again we write an allowed elementary $m$-path as
\[(a;s_1,s_2,\ldots,s_m),\quad s_i\in\{1,s\}.\]
Similarly, we have the following result.
\begin{itemize}
    \item If $(s_j,s_{j+1})=(1,1)$ for some $1\le j\le m-1$, then the internal face obtained by deleting the $j$-th vertex
has, at position $j$, a merged increment $2$.
Moreover this illegal $(m-1)$-face cannot arise as an internal face of any other allowed elementary $m$-path.
\item If $(s_j,s_{j+1})=(s,s)$ for some $1\le j\le m-1$, then the analogous statement holds with merged increment $2s$.
\end{itemize}
Assume $P\in\Omega_m(\vec{C}_n^S)$, $m\ge3$, then every elementary appearing in $P$ is forced to be of the form
\[(a;1,s,1,s,\ldots)\quad\text{or}\quad (a;s,1,s,1,\ldots).\]
Now let us investigate the $j=1$ internal deletion face. 
\begin{itemize}
    \item For elementary path $(a;1,s,1,s,\ldots)$, deleting the first internal vertex merges $(1,s)$ into the illegal path $(a;s+1,1,s,\ldots)$.
    \item For elementary path $(a;s,1,s,1,\ldots)$, deleting the first internal vertex merges $(s,1)$ into the illegal path $(a;s+1,s,1,\ldots)$.
\end{itemize}
The resulting illegal faces are distinct. Thus, their coefficients in $\partial P$ must vanish, forcing the coefficients of both kinds of elementary path in $P$ to be zero. Hence $P=0$, which means that 
\[\Omega_m(\vec{C}_n^S)=0, \quad m\geq3.\qedhere\]
\end{proof}

\begin{theorem}\label{thm:1s}
For the circulant digraph $\vec{C}_n^S$ with $n\ge5$, $S=\{1,s\}$, $1<s<n/2$. 
\begin{itemize}
    \item[(1)] If $s=2$, then
    \[H_0^{\mathrm{path}}\cong\K,\quad H_1^{\mathrm{path}}\cong\K,\quad H_m^{\mathrm{path}}\cong0,~m\geq2.\]
    \item[(2)] If $s\neq2$, then
    \[H_0^{\mathrm{path}}\cong\K,\quad H_1^{\mathrm{path}}\cong\K^2,\quad H_2^{\mathrm{path}}\cong\K,\quad H_m^{\mathrm{path}}\cong0,~m\geq3.\]
\end{itemize}
\end{theorem}

\begin{proof}
(1) It is similar to the proof of Proposition \ref{prop:example-C512}. Using the Fourier modes, one can reduce the rank computation to the $\tau$-eigenspaces. Then we have
\[\rank(\partial_1)=n-1,\qquad \rank(\partial_m)=n,~~ m\geq2.\]
It follows that 
\[\dim H_1^{\mathrm{path}}=(2n-(n-1))-n=1,\quad H_{m\geq2}^{\mathrm{path}}=0.\]

(2) For $s\neq2$, the path complex of $\vec{C}_n^{S}$ is just
\[0~\rightarrow~\Omega_2~\xrightarrow{\partial_2}~\Omega_1~\xrightarrow{\partial_1}~\Omega_0~\rightarrow ~0.\]
A Fourier diagonalization (or an explicit circulant matrix computation) gives $\rank(\partial_2)=n-1$.
Hence 
\[\dim H_2^{\mathrm{path}}=1,\qquad \dim H_1^{\mathrm{path}}=(2n-(n-1))-(n-1)=2.\qedhere\]
\end{proof}

\begin{remark} An easier proof is available using the results in the following subsections. For $s=2$, apply Theorem \ref{thm:S12d} with $d=2$. For $s\ne2$, the three unordered
pair sums $2,1+s,2s$ are distinct and avoid $S$, so Theorem~\ref{st:exterior} gives
\[
H_*^{\mathrm{path}}(\vec C_n^{\{1,s\}};\K)
\cong\Lambda^*(\K^2).
\]
\end{remark}

\subsubsection{The case $S=\{1,2,\ldots, d\}$, $2<d<n/2$}

This interval case could be computed explicitly using an infinite lift. The proof
uses the classical homotopy invariance of path homology
\cite{GLMY:homotopy} and the comparison theorem for projective resolutions \cite[Theorem 2.2.6]{Weibel}.

\begin{lemma}\label{lem:interval-lift}
Let $d\geq1$, let $X_d$ be the digraph with vertex set $\mathbb Z$ and an arrow
$x\to y$ precisely when $1\leq y-x\leq d$. Then all chains have finite
support. Over any field $\K$,
\[
\widetilde H_m^{\mathrm{path}}(X_d;\K)=0\qquad(m\geq0).
\]
Moreover, if $R=\K[t,t^{-1}]$ acts by translation of every vertex by
one, then the augmented complex
\[
\cdots\longrightarrow\Omega_1(X_d;\K)
\longrightarrow\Omega_0(X_d;\K)\longrightarrow\K\longrightarrow0
\]
is a free $R$-resolution of the trivial module $\K=R/(t-1)$.
\end{lemma}
\begin{proof}
Every finite chain is supported in an induced subgraph $X_d[A,B]$. For
$A\leq k\leq B$, define
\[
f_k(x)=\max\{x,k\}.
\]
If $x\to y$, then either $f_k(x)=f_k(y)$ or
$1\leq f_k(y)-f_k(x)\leq d$, so $f_k$ is a digraph map. Moreover,
$f_A=\mathrm{id}$, $f_B\equiv B$, and $f_k(x)\to f_{k+1}(x)$ whenever
the two values differ. Thus $(f_k)_{k=A}^B$ contracts $X_d[A,B]$ to
$B$. Hence every finite reduced cycle bounds in a finite interval, proving the acyclic property.

For freeness, for an allowed elementary path
$(a;s_1,\ldots,s_m)$, with $a\in\mathbb Z$ and
$s_i\in\{1,\ldots,d\}$, the only forbidden faces are inner faces for which an adjacent sum exceeds $d$. Cancellation of such a face involves only paths with the same initial
vertex and the same remaining step word. Thus there is a finite vector space $V_m(d)\subseteq\K^{\{1,\ldots,d\}^m}$ of step-word coefficients satisfying all these cancellation conditions, and
\[
\Omega_m(X_d;\K)=R\otimes_\K V_m(d).
\]
For $m=0$, set $V_0(d)=\K$. These modules are free; the augmentation is $R$-linear with trivial action on $\K$. Exactness follows from
the acyclicity just proved.
\end{proof}

\begin{theorem}\label{thm:S12d}
Let $1\leq e\leq d$ and $n>2d$. Write $G_j=\vec C_n^{1,\ldots,j}$.
For every field $\K$,
\[
H_m^{\mathrm{path}}(G_d;\K)\cong
\begin{cases}
\K,&m=0,1,\\
0,&m\geq2.
\end{cases}
\]
The natural inclusion $i:\Omega_*(G_e;\K)\hookrightarrow
\Omega_*(G_d;\K)$ induces an isomorphism in every degree. Furthermore,
there are a chain map $\Pi:\Omega_*(G_d;\K)\to\Omega_*(G_e;\K)$
and a degree-one linear map $H$ such that
\[
\Pi i=\mathrm{Id},\qquad
\mathrm{Id}-i\Pi=\partial H+H\partial,\qquad Hi=0.
\]
The generator of $H_1^{\mathrm{path}}$ can be represented by the unit-step cycle
\[
c_n=\sum_{a\in\mathbb Z_n}(a;1).
\]
\end{theorem}
\begin{proof}
Set $R_n=R/(t^n-1)$. Since $n>2d$, every sum of two adjacent
steps lies in $[2,2d]\subset(0,n)$. Reduction modulo $n$ therefore
preserves whether a merged step belongs to $\{1,\ldots,d\}$, and no
merged step becomes zero. Hence the forbidden-face equations and the
boundary maps agree before and after reduction. In each degree $m$, reduction
of the initial vertex identifies
\[
R_n\otimes_R\Omega_m(X_d;\K)
=R_n\otimes_{\K}V_m(d)
\quad\text{with}\quad
\Omega_m(G_d;\K).
\]
These degreewise identifications form an isomorphism of chain complexes
\begin{equation}\label{eq:interval-base-change}
R_n\otimes_R\Omega_*(X_d;\K)
\cong \Omega_*(G_d;\K).
\end{equation}

By Lemma~\ref{lem:interval-lift}, the homology of
\eqref{eq:interval-base-change} is
$\operatorname{Tor}^R_*(R_n,\K)$. It can be computed using the
length-one free resolution
\[
0\longrightarrow R\xrightarrow{\ t-1\ }R
\longrightarrow\K\longrightarrow0.
\]
Multiplication by $t-1$ on $R_n$ has one-dimensional kernel spanned by
$1+t+\cdots+t^{n-1}$ and cokernel given by $\K$.
Then we obtain the above homology result for arbitrary field.
The comparison map from this length-one resolution to
$\Omega_*(X_d)$ sends its degree-zero generator to $(0)$ and its
degree-one generator to $(0;1)$. Thus, the generator of $H_1^{\mathrm{path}}$ can be represented by $c_n$.

The inclusions $X_e\hookrightarrow X_d$ give maps of the free
resolutions in Lemma~\ref{lem:interval-lift} lifting the identity
of $\K$. By the comparison theorem for projective resolutions,
these maps are $R$-linear chain-homotopy equivalences. Tensoring
the maps and their homotopies with $R_n$ shows that $i$ induces
an isomorphism on homology.

Since $i$ is an injective quasi-isomorphism of complexes of
$\mathbb K$-vector spaces, the quotient complex
\[
Q_*=\Omega_*(G_d;\mathbb K)/\Omega_*(G_e;\mathbb K)
\]
is contractible. The standard splitting argument therefore gives
maps $\Pi$ and $H$ such that
\[
\Pi i=\mathrm{Id},\qquad
\mathrm{Id}-i\Pi=\partial H+H\partial,\qquad Hi=0.\qedhere
\]
\end{proof}

\begin{remark}
For $e=1$, the retraction is explicit:
\[
\Pi_0(a)=(a),\qquad
\Pi_1(a;s)=\sum_{j=0}^{s-1}(a+j;1),\qquad
\Pi_m=0\quad(m\ge2).
\]
Thus a step of length $s$ is replaced by the corresponding chain of
$s$ unit steps. The chain-map identity follows by telescoping, and the
preceding theorem supplies a homotopy relative to $G_1=\vec{C}_n^1$.
\end{remark}

\section{The stable chain complex and the infinite cyclic lift}
\label{st:section}

In this section, we use the translation homotopies first to identify the invariant step-word model in characteristic zero. Then we use the infinite cyclic lift to prove the same
all-degree stability over an arbitrary field. 

\subsection{Translations act trivially on homology}

Let $H$ be an abelian group, $S\subset H\setminus\{e\}$, and
$G=\Cay(H,S)$. For $g\in H$, write $T_g$ for translation by $g$ on
vertices and on chains, which means 
\[
T_g(a;s_1,\ldots,s_m)=(a+g;s_1,\ldots,s_m)=e_{a+g,a+g+s_1,\ldots,a+g+s_1+\cdots+s_m}.
\]

\begin{lemma}\label{st:prism}
For $s\in S$, define
\begin{equation}\label{st:h-formula}
 h_s(a;s_1,\ldots,s_m)
 =\sum_{i=0}^{m}(-1)^i
 (a;s_1,\ldots,s_i,s,s_{i+1},\ldots,s_m).
\end{equation}
Then $h_s$ induces maps $\Omega_m(G)\to\Omega_{m+1}(G)$ satisfying
\begin{equation}\label{st:prism-identity}
 \partial h_s+h_s\partial=T_s-\mathrm{Id}.
\end{equation}
The maps $h_s$ commute with all translations. Consequently $T_s$ induces
the identity on path homology.
\end{lemma}

\begin{proof}
Write $p=e_{x_0,\ldots,x_m}$ and set
\[
p_i=e_{x_0,\ldots,x_i,x_i+s,x_{i+1}+s,\ldots,x_m+s},
\qquad 0\leq i\leq m.
\]
Then $h_s p=\sum_{i=0}^m(-1)^ip_i$. Since translation by $s$ preserves edges and $x_i\to x_i+s$ is an edge, each $p_i$ is allowed whenever $p$ is allowed.

Expanding the alternating boundaries of the $p_i$, the terms obtained
by deleting an original vertex cancel with the corresponding terms in
$h_s\partial p$, while the terms at the transition from $x_i$ to
$x_i+s$ cancel successively as $i$ varies. The only remaining terms
are $T_sp$ and $-p$. Hence, on the regular path complex,
\[
\partial h_s+h_s\partial=T_s-\mathrm{Id}.
\]

Now let $u\in\Omega_m(G)$. Both $h_su$ and $h_s\partial u$ are allowed,
and the identity above gives
\[
\partial h_su=(T_s-\mathrm{Id})u-h_s\partial u,
\]
which is also allowed. Thus $h_su\in\Omega_{m+1}(G)$. By definition, $h_s$ commutes with translations.
\end{proof}

\begin{theorem}\label{st:invariants}
Let $H$ be finite abelian, suppose that $S$ generates $H$ as a group,
and assume that $|H|$ is invertible in $\K$.
Set
\[
 P=\frac1{|H|}\sum_{g\in H}T_g,
 \qquad \Omega_*(G)^H=\{u:T_gu=u\text{ for every }g\in H\}.
\]
The inclusion $\iota:\Omega_*(G)^H\hookrightarrow\Omega_*(G)$ and the
averaging map $P:\Omega_*(G)\to\Omega_*(G)^H$ are chain homotopy inverse
maps. In particular,
\begin{equation}\label{st:invariant-homology}
 H_*^{\mathrm{path}}(G;\K)\cong H_*(\Omega_*(G)^H,\partial).
\end{equation}
Over a splitting field, every nontrivial character subcomplex is
contractible.
\end{theorem}

\begin{proof}
For a negative generator use
$h_{-s}=-T_{-s}h_s$, so that
$\partial h_{-s}+h_{-s}\partial=T_{-s}-\mathrm{Id}$.
For each $g$ choose an expression
$g=\epsilon_1s_1+\cdots+\epsilon_rs_r$, where $s_i\in S$ and
$\epsilon_i\in\{1,-1\}$. Set $g_0=0$ and
$g_j=\epsilon_1s_1+\cdots+\epsilon_js_j$. Then
\[
 H_g=\sum_{j=1}^r T_{g_{j-1}}h_{\epsilon_js_j}
 \quad\text{satisfies}\quad
 \partial H_g+H_g\partial=T_g-\mathrm{Id}.
\]
Thus
\[
 D=-\frac1{|H|}\sum_{g\in A}H_g,
 \qquad \partial D+D\partial=\mathrm{Id}-\iota P,
 \qquad P\iota=\mathrm{Id}.
\]
For the character assertion, choose $s\in S$ with $\chi(s)\ne1$.
On the $\chi$-subcomplex, \eqref{st:prism-identity} reads
\[
\partial h_s+h_s\partial=(\chi(s)-1)\mathrm{Id}.
\]
The operator $(\chi(s)-1)^{-1}h_s$ is therefore a contraction.
\end{proof}

For a circulant with $1\in S$, let $\tau=T_1$. The averaging homotopy has the explicit form
\begin{equation}\label{st:cyclic-average}
 D=-\frac1n\sum_{a=1}^{n-1}\sum_{j=0}^{a-1}\tau^jh_1,
 \qquad
 \partial D+D\partial=\mathrm{Id}-\frac1n\sum_{a=0}^{n-1}\tau^a.
\end{equation}
Consequently, if $\tau u=\lambda u$ and $\lambda\ne1$, then
$h_1/(\lambda-1)$ contracts the corresponding Fourier block.
The same argument applies without assuming $1\in S$: since $S$ generates
$\mathbb Z_n$, every nontrivial character satisfies $\lambda^s\ne1$ for
some $s\in S$, and $h_s/(\lambda^s-1)$ contracts the corresponding
Fourier block.

\begin{remark}
In characteristic zero, the GLMY translation homotopy contracts all nontrivial Fourier blocks, reducing the computation to invariant chains. When $\operatorname{char}\mathbb K\mid |H|$, this reduction need not hold; arbitrary-field stability is instead proved below using the cyclic lift.
\end{remark}

\subsection{The step complex associated with $S$}

Fix a finite set $S\subset\mathbb Z_{>0}$ containing $1$, and set
$M=\max S$, $V_S=\K^S$, with basis $[s]$ for $s\in S$.
The following construction uses ordinary integer addition.
Set $F_S=(S+S)\setminus S$ and define
\[
 \mu_F:V_S\otimes V_S\longrightarrow\K^{F_S},\qquad
 \mu_F([s]\otimes[t])=
 \begin{cases}[s+t],&s+t\in F_S,\\0,&s+t\in S.\end{cases}
\]
Write $[s_1|\cdots|s_m]$ for an elementary tensor.

\begin{definition}\label{st:complex-definition}
Let $C_0(S)=\K$, $C_1(S)=V_S$, and, for $m\ge2$, let
\begin{equation}\label{st:Cm}
 C_m(S)=\bigcap_{j=1}^{m-1}
\ker\bigl(\mathrm{Id}^{\otimes(j-1)}\otimes\mu_F
\otimes\mathrm{Id}^{\otimes(m-j-1)}\bigr)
 \ \subseteq V_S^{\otimes m}.
\end{equation}
On $C_m(S)$, the differential is the restriction of
\begin{align}\label{st:delta}
 \delta[s_1|\cdots|s_m]
 ={}&[s_2|\cdots|s_m]
       +(-1)^m[s_1|\cdots|s_{m-1}]\notag\\
 &+\sum_{\substack{1\le j<m\\s_j+s_{j+1}\in S}}
 (-1)^j[s_1|\cdots|s_j+s_{j+1}|\cdots|s_m].
\end{align}
For $m=1$ the two endpoint terms are scalars and cancel, so
$\delta:C_1(S)\to C_0(S)$ is zero.
\end{definition}

\eqref{st:Cm} says that for every $j$,
every fixed prefix and suffix, and every $r\in F_S$, the coefficients satisfy
\begin{equation}\label{st:local-relations}
 \sum_{\substack{s,t\in S\\s+t=r}}
 c_{s_1,\ldots,s_{j-1},s,t,s_{j+2},\ldots,s_m}=0.
\end{equation}

\begin{theorem}\label{st:stability}
For every integer $n>2M$, there are translation-equivariant vector-space
isomorphisms
\begin{equation}\label{st:free-module}
 \Omega_m(\vec C_n^S;\K)\cong\K[\mathbb Z_n]\otimes C_m(S)
 \qquad(m\ge0).
\end{equation}
The map
\begin{equation}\label{st:Phi}
 \Phi_n([s_1|\cdots|s_m])
 =\frac1n\sum_{a\in\mathbb Z_n}(a;s_1,\ldots,s_m)
\end{equation}
restricts to a chain isomorphism
\[
 (C_*(S),\delta)\xrightarrow{\ \cong\ }
 (\Omega_*(\vec C_n^S;\K)^{\mathbb Z_n},\partial).
\]
In particular $\delta C_m(S)\subseteq C_{m-1}(S)$, $\delta^2=0$, and
\begin{equation}\label{st:stable-homology}
 H_m^{\mathrm{path}}(\vec C_n^S;\K)
 \cong H_m(C_*(S),\delta)
 \qquad(n>2M,\ m\ge0).
\end{equation}
\end{theorem}

\begin{proof}
Every allowed path is uniquely written as $(a;s_1,\ldots,s_m)$. An
inner face replaces $(s_j,s_{j+1})$ by $s_j+s_{j+1}$. Since
$0<s_j+s_{j+1}<n$, it is allowed only when
$s_j+s_{j+1}\in S$, and distinct positions give distinct forbidden
faces. Thus boundary invariance is exactly
\eqref{st:local-relations}, then
proving \eqref{st:free-module}.

Translation-invariant chains are precisely the averages
\eqref{st:Phi}. Under this identification, the endpoint faces give
the first two terms of \eqref{st:delta}, the allowed inner faces give
its merging terms, and the forbidden faces cancel by
\eqref{st:local-relations}. Hence
\[
\partial\Phi_n=\Phi_n\delta.
\]
The remaining assertions follow from $\partial^2=0$ and
Theorem~\ref{st:invariants}.
\end{proof}

\begin{proposition}
\label{st:additive-types}
Suppose $S,S'\subset\mathbb Z_{>0}$ are finite, $1\in S,S'$, and
$\varphi:S\to S'$ is a bijection such that
\begin{align*}
 s+s'=u &\quad\Longleftrightarrow\quad
\varphi(s)+\varphi(s')=\varphi(u),\\
s+s'=u+v &\quad\Longleftrightarrow\quad
\varphi(s)+\varphi(s')=\varphi(u)+\varphi(v)
\end{align*}
for all steps appearing in these equations. Then entrywise relabeling gives a chain isomorphism
$C_*(S)\cong C_*(S')$.
Thus the stable path complex depends only on the short additive type of step set.
\end{proposition}

\begin{proof}
The relabeling
\[
[s_1|\cdots|s_m]\longmapsto
[\varphi(s_1)|\cdots|\varphi(s_m)]
\]
preserves the local relations defining $C_m(S)$ by the second
assumption, and commutes with the differential by the first.
Thus it defines a chain isomorphism
$C_*(S)\cong C_*(S')$. For $d$ labeled steps, the additive type is
determined by finitely many relations of the forms
$s+s'=u$ and $s+s'=u+v$, which proves the finiteness assertion.
\end{proof}

\subsection{The infinite cyclic lift and stability in arbitrary characteristic}
\label{subsec:cyclic-lift}

To prove stability over an arbitrary field $\K$, consider the infinite cyclic lift
\[
X_S=\operatorname{Cay}(\mathbb Z,S),\qquad
R=\K[t,t^{-1}],\qquad
F_*=\Omega_*(X_S;\K),
\]
where $t$ translates every vertex by one and chains have finite support. In this subsection, again $S\subset\mathbb Z_{>0}$ is finite, $1\in S$.

\begin{lemma}\label{cl:free}
There are degreewise $R$-module isomorphisms
\[
 F_m\cong R\otimes_\K C_m(S).
\]
For $n>2\max S$, reduction of the initial vertex modulo $n$ induces a
chain isomorphism
\begin{equation}\label{cl:quotient}
 F_* /(t^n-1)F_*\cong\Omega_*(\vec C_n^S;\K).
\end{equation}
There is also a chain isomorphism
\begin{equation}\label{cl:augmentation}
 F_* /(t-1)F_*\cong(C_*(S),\delta).
\end{equation}
\end{lemma}

\begin{proof}
Every allowed path in the infinite lift is uniquely written as
\[
t^a[s_1|\cdots|s_m],\qquad a\in\mathbb Z,\quad s_i\in S.
\]
At each initial vertex, boundary invariance is exactly 
\eqref{st:local-relations}. Hence
\[
F_m\cong R\otimes_{\mathbb K}C_m(S).
\]
When $n>2\max S$, reduction of $a$ modulo $n$ preserves these local relations and gives the underlying graded isomorphism in \eqref{cl:quotient}.

In the resulting module coordinates, the boundary, restricted to $F_m$, is
\begin{align*}
 d[s_1|\cdots|s_m]
 ={}&t^{s_1}[s_2|\cdots|s_m]
       +(-1)^m[s_1|\cdots|s_{m-1}]\\
 &+\sum_{\substack{1\leq j<m\\s_j+s_{j+1}\in S}}
 (-1)^j[s_1|\cdots|s_j+s_{j+1}|\cdots|s_m].
\end{align*}
Then the quotient maps commute with the boundary, and setting
$t=1$ gives exactly \eqref{st:delta}. Thus, in every characteristic, $\delta$ preserves $C_*(S)$ and squares to zero.
\end{proof}

\begin{theorem}\label{cl:ses}
Let $\K$ be any field and $n>2\max S$. For every $m\geq0$, there is
a split short exact sequence
\begin{equation}\label{cl:ses-equation}
 0\longrightarrow H_m(X_S;\K)
 \xrightarrow{(q_n)_*}H_m(\vec C_n^S;\K)
 \xrightarrow{\partial_n}H_{m-1}(X_S;\K)
 \longrightarrow0,
\end{equation}
where $H_{-1}=0$ and $q_n$ denotes reduction modulo $n$. Write
\[
 p_n(t)=1+t+\cdots+t^{n-1},\qquad K_n=p_n(t)h_1,
\]
with $h_1$ in Lemma~\ref{st:prism} on $X_S$.
A specified splitting is
\begin{equation}\label{cl:splitting}
[z]\longmapsto[q_nK_nz],\qquad z\in F_{m-1},\quad dz=0.
\end{equation}
In particular,
\begin{equation}\label{cl:decomposition}
 H_m(\vec C_n^S;\K)
 \cong H_m(X_S;\K)\oplus H_{m-1}(X_S;\K).
\end{equation}
\end{theorem}

\begin{proof}
The standard chain-homotopy operator is $R$-linear and satisfies
$dh_1+h_1d=t-1$. Therefore
\[
 dK_n+K_nd=t^n-1.
\]
Every $F_m$ is free over the integral domain $R$, so multiplication by
the nonzero Laurent polynomial $t^n-1$ is injective. Lemma~\ref{cl:free}
therefore supplies a short exact sequence of chain complexes
\[
 0\longrightarrow F_*\xrightarrow{\ t^n-1\ }F_*
 \xrightarrow{\ q_n\ }\Omega_*(\vec C_n^S;\K)\longrightarrow0.
\]
The first map induces zero on homology because of the 
null-homotopy. The long exact sequence breaks into
\eqref{cl:ses-equation}. We use the connecting-map convention
$\partial_n[w]=[z]$ when a lift $\widetilde w$ satisfies
$d\widetilde w=(t^n-1)z$.

For a cycle $z$, the identity $dK_nz=(t^n-1)z$ shows that $q_nK_nz$
is a cycle whose connecting image is $[z]$. If $z=dw$, then
$q_nK_ndw=-d(q_nK_nw)$, so \eqref{cl:splitting} is well defined
on homology. It is a right inverse to $\partial_n$, proving the split
decomposition without division by $n$.
\end{proof}

\begin{theorem}
\label{cl:stability}
Let $\K$ be any field and $n>2\max S$. For every $m\geq0$, there are
isomorphisms
\begin{equation}\label{cl:stability-equation}
 H_m(\vec C_n^S;\K)\cong H_m(C_*(S;\K),\delta)
 \cong H_m(X_S;\K)\oplus H_{m-1}(X_S;\K).
\end{equation}
Consequently, for each fixed coefficient field, the entire Betti
sequence is independent of $n$ in this range.
\end{theorem}

\begin{proof}
Apply the proof of Theorem~\ref{cl:ses} to multiplication by $t-1$, using the algebraic quotient
\eqref{cl:augmentation} and the homotopy $K_1=h_1$.
This gives the second isomorphism in
\eqref{cl:stability-equation}; comparison with
\eqref{cl:decomposition} gives the first.

More explicitly, set $D_m=F_m\oplus F_{m-1}$ with differential
$d_D(x,y)=(dx,-dy)$. For $j=1,n$, the maps
\[
 \Psi_j:D_*\longrightarrow F_* /(t^j-1)F_*,
 \qquad \Psi_j(x,y)=q_j(x+K_jy)
\]
are chain maps. Their induced homology maps are precisely the split
isomorphisms just proved. Thus the stable comparison is realized by the
explicit zigzag of quasi-isomorphisms
\[
\Omega_*(\vec C_n^S;\K)
\xleftarrow{\ \Psi_n\ }D_*
\xrightarrow{\ \Psi_1\ }(C_*(S;\K),\delta).\qedhere
\]
\end{proof}

\begin{corollary}
\label{cl:lift-betti}
All groups $H_m(X_S;\K)$ are finite dimensional. For a digraph $Y$, put
\[
P_Y(z)=\sum_{m\geq0}\dim_\K H_m(Y;\K)z^m.
\]
Then
\begin{equation}\label{cl:poincare}
 P_{\vec C_n^S}(z)=(1+z)P_{X_S}(z)
 \qquad(n>2\max S).
\end{equation}
If $H_m(\vec C_n^S;\mathbb K)=0$ for all sufficiently large $m$, then $\chi(\vec C_n^S;\mathbb K)=0$. 
\end{corollary}

\begin{proof}
The injection in \eqref{cl:ses-equation} embeds $H_m(X_S;\K)$ in the homology of a finite-dimensional chain group. Taking dimensions in the split sequence gives \eqref{cl:poincare}. Under the
stated vanishing assumption, evaluation at $z=-1$ gives the Euler
characteristic formula. 
\end{proof}

\subsection{Further consequences of the stable complex}

We first record a presentation of $H_1^{\mathrm{path}}$ over an arbitrary field.

\begin{proposition}
\label{st:H1}
Let $S\subset\mathbb Z_{>0}$ be finite with $1\in S$, let $n>2\max S$,
and let $\K$ be any field.
Let $R_S\subseteq V_S$ be spanned by
\begin{equation}
\label{st:collision-relations}
\begin{aligned}
 &[s]+[t]-[s+t] && (s,t,s+t\in S),\\
 &[s]+[t]-[u]-[v] && (s+t=u+v\notin S).
\end{aligned}
\end{equation}
Then
\[H_1^{\mathrm{path}}(\vec C_n^S;\K)\cong V_S/R_S.\]
\end{proposition}

\begin{proof}
By Definition~\ref{st:complex-definition},
\[
\ker(\delta:C_1(S)\to C_0(S))=V_S,
\qquad
\operatorname{im}(\delta:C_2(S)\to C_1(S))=R_S.
\]
Hence \(H_1(C_*(S),\delta)\cong V_S/R_S\), and the result follows from
Theorem~\ref{cl:stability}.
\end{proof}

\vskip 0.1cm
The following homology formula is a special case of
\cite[Theorem~4.11]{DIMZ}; we include a direct proof to obtain the invariant-chain basis and explicit representatives used below.

\begin{theorem}
\label{st:exterior}
Suppose $S\subset\mathbb Z_{>0}$ contains $1$ and satisfies
\begin{enumerate}
 \item $(S+S)\cap S=\varnothing$;
 \item $s+t=u+v$ for $s,t,u,v\in S$ implies equality of the unordered
 pairs $\{s,t\}=\{u,v\}$, counted with multiplicity.
\end{enumerate}
Write $d=|S|$. For $n>2\max S$, there are graded vector-space
isomorphisms
\begin{equation}\label{st:exterior-formula}
 H_*^{\mathrm{path}}(\vec C_n^S;\K)
 \cong\Lambda^*(V_S),\qquad
 \beta_m=\binom dm.
\end{equation}
When $n$ is invertible in $\K$, an explicit basis of homology is
represented, for $s_1<\cdots<s_m$ in $S$, by
\begin{equation}\label{st:alternating-cycles}
 Z_{s_1,\ldots,s_m}^{(n)}
 =\frac1n\sum_{a\in\mathbb Z_n}\sum_{\sigma\in\mathfrak S_m}
 \operatorname{sgn}(\sigma)
 (a;s_{\sigma(1)},\ldots,s_{\sigma(m)}).
\end{equation}
In this case, $\Omega_m(\vec C_n^S)=0$ for $m>d$.
\end{theorem}

\begin{proof}
By the two assumptions, the local relations
\eqref{st:local-relations} are precisely the alternating relations:
interchanging two adjacent steps changes the sign, while a tensor with
two equal adjacent steps vanishes. Hence for $0<m\le d$,
\[
C_m(S)\cong\Lambda^m(V_S),
\]
with basis
\[
\sum_{\sigma\in\mathfrak S_m}\operatorname{sgn}(\sigma)
[s_{\sigma(1)}|\cdots|s_{\sigma(m)}],\qquad s_1<\cdots<s_m,
\]
and \(C_m(S)=0\) for \(m>d\).

Since \((S+S)\cap S=\varnothing\), the differential has no interior
terms, and its two endpoint terms cancel on these alternating tensors.
Thus \(\delta=0\). The homology formula now follows from Theorem \ref{cl:stability}, and the representatives
\eqref{st:alternating-cycles} follow from \eqref{st:Phi}.
Finally, \eqref{st:free-module} gives
\(\Omega_m(\vec C_n^S)=0\) for \(m>d\).
\end{proof}

\begin{example}\label{st:geometric-exterior}
For $B\ge3$, the set
\[
 S=\{1,B,B^2,\ldots,B^{d-1}\}
\]
has no legal pair sums, and base-$B$ expansion makes every unordered pair sum unique. Hence Theorem~\ref{st:exterior} gives
\[
 H_*^{\mathrm{path}}(\vec C_n^S;\K)\cong\Lambda^*(\K^d)
 \qquad(n>2B^{d-1}).
\]
Thus stable path homology can be nonzero in arbitrarily high degree as the number of steps grows.
\end{example}

The stable complex is therefore determined by the short additive
relations among the steps. In the next section we encode these relations in
a lattice lift and apply it to connection sets with three positive steps.

\section{The short-relation lift and the three-step classification}
\label{sec:three-step-classification}

The stable step-word complex depends only on short additive relations. We
first encode those relations in a lattice lift and then determine all
possible lifts for $S=\{1,a,b\}$ with $1<a<b$.

\subsection{Freiman admissibility and the short-relation lift}
\label{subsec:acyclic-lift}

Let $S\subset\mathbb Z_{>0}$ be finite with $1\in S$.
\begin{definition}\label{cl:freiman}
Set $A=\{0\}\cup S$. We call a positive integer $n$
\emph{admissible for $S$} if reduction
modulo $n$ is injective on $A+A$. Equivalently, it preserves and reflects
all relations
\[
 a+b=c+d,\qquad a,b,c,d\in A.
\]
\end{definition}
The inclusion of $0$ records legal merges as well as pair-sum
collisions. The bound $n>2\max S$ is sufficient for admissibility.

A \emph{short additive relation} on a set of positive integers $S$ is an
equality between sums of at most two elements of $S$, with the empty
sum interpreted as zero. Let $\mathbb Z^S$ have basis $(E_s)_{s\in S}$,
and set
\[
 L_S=\left\langle
 E_s+E_t-E_u\ (s+t=u\in S),\quad
 E_s+E_t-E_u-E_v\ (s+t=u+v)
 \right\rangle_{\mathbb Z},
 \qquad Q_S=\mathbb Z^S/L_S.
\]
Write $\bar E_s$ for the image of $E_s$. There is a homomorphism
$\ell:Q_S\to\mathbb Z$, $\ell(\bar E_s)=s$; in particular, all the
$\bar E_s$ are distinct and nonzero. Reduction of $\ell$ modulo $n$
gives a surjection $\phi_n:Q_S\to\mathbb Z_n$ whenever $1\in S$.
Set
\[
 \widetilde G_S=\Cay(Q_S,\{\bar E_s:s\in S\}),
 \qquad N_n=\ker\phi_n.
\]
For an admissible $n$ in Definition~\ref{cl:freiman}, in particular
for $n>2\max S$, equality modulo $n$ between two sums of at most two
steps is equivalent to equality as integers. Such an equality holds
in $Q_S$ precisely when it is a short additive relation: one direction
is built into $L_S$, and the other follows by applying $\ell$.
Consequently $\widetilde G_S\to\vec C_n^S$ preserves all short-path
endpoint identifications, and is a regular $2$-covering in the sense of
\cite{DIMZ}.

\begin{proposition}
\label{prop:acyclic-lift-criterion}
Suppose $1\in S$, $n$ is admissible for $S$, $Q_S\cong\mathbb Z^r$, and
\[
 \widetilde H_m^{\mathrm{path}}(\widetilde G_S;\K)=0
 \qquad(m\geq0).
\]
Then, as graded vector spaces,
\[
 H_m^{\mathrm{path}}(\vec C_n^S;\K)
 \cong H_m(N_n;\K)
 \cong \Lambda^m(\K^r).
\]
\end{proposition}

\begin{proof}
This is the acyclic-cover mechanism of \cite[Theorem 3.23]{DIMZ}; the following direct argument also makes the
coefficient modules explicit. In step coordinates, the condition
$\partial u\in\A_{m-1}$ consists exactly of cancellation of the
forbidden adjacent merges. For a fixed initial vertex, these
conditions depend only on the short additive relations. Let $W_m$
be their solution space in the vector space with basis $S^m$; set
$W_0=\K$. Therefore
\[
 \Omega_m(\widetilde G_S)\cong\K[Q_S]\otimes_\K W_m,
 \qquad
 \Omega_m(\vec C_n^S)\cong\K[\mathbb Z_n]\otimes_\K W_m.
\]
These are module identifications; the boundary includes the usual
translation in its first face. Since $Q_S/N_n\cong\mathbb Z_n$,
$\K[Q_S]$ is a free $\K[N_n]$-module, and quotienting by $N_n$
identifies the two chain complexes:
\[
 \K\otimes_{\K[N_n]}\Omega_*(\widetilde G_S)
 \cong\Omega_*(\vec C_n^S).
\]
By the assumed acyclicity, the augmented complex on the left before
tensoring is a free resolution of the trivial $\K[N_n]$-module $\K$.
It follows that its coinvariant homology is $H_*(N_n;\K)$.
The finite-index subgroup $N_n$ of $\mathbb Z^r$ is itself free abelian
of rank $r$. The Koszul resolution over
$\K[t_1^{\pm1},\ldots,t_r^{\pm1}]$, for the regular sequence
$t_1-1,\ldots,t_r-1$, gives the stated exterior-power dimensions.
\end{proof}

\subsection{The seven additive patterns}

\begin{theorem}
\label{thm:complete-three-step}
Let $S=\{1,a,b\}$ with integers $1<a<b$, let $n>2b$, and let
$\K$ be any field.
More generally, the same conclusions hold for every modulus admissible
for $S$ in Definition~\ref{cl:freiman}.
Define
\[
 r(a,b)=
 \begin{cases}
  1,&(a,b)=(2,3)\text{ or }(2,4),\\
  3,&a\geq3\text{ and }b\notin\{a+1,2a-1,2a\},\\
  2,&\text{otherwise}.
 \end{cases}
\]
Then, for every $m\geq0$,
\[
H_m^{\mathrm{path}}(\vec C_n^{1,a,b};\K)
\cong\Lambda^m(\K^{r(a,b)}),
\qquad \beta_m=\binom{r(a,b)}m.
\]
In particular, all groups in degrees above three vanish, and
$H_3\neq0$ holds only in the case $r(a,b)=3$.
\end{theorem}

\begin{proof}
Set $X=E_1$, $Y=E_a$, and $Z=E_b$.
The possible sums of unordered pairs are
\[
 2,\quad 1+a,\quad 1+b,\quad 2a,\quad a+b,\quad 2b.
\]
The only ones that can be steps are $2=a$, $1+a=b$, and $2a=b$.
Two different unordered pairs can have the same sum only when
$1+b=2a$: all other comparisons are excluded by $1<a<b$.
This yields exactly the seven mutually exclusive cases in
Table~\ref{tab:three-step-patterns}. We now prove that every lifted digraph in the table has zero reduced path homology.

\begin{table}[h]
\centering
\small
\renewcommand{\arraystretch}{1.25}
\begin{tabular}{p{0.24\textwidth}p{0.23\textwidth}p{0.34\textwidth}c}
\hline
Condition & Relations generating $L_S$ & Lifted steps in $\mathbb Z^r$ & $r$\\
\hline
$a\geq3$, $b\notin\{a+1,2a-1,2a\}$
 & None & $e_1,e_2,e_3$ & $3$\\
$a=2$, $b\geq5$
 & $Y=2X$ & $(1,0),(2,0),(0,1)$ & $2$\\
$a\geq3$, $b=a+1$
 & $Z=X+Y$ & $(1,0),(0,1),(1,1)$ & $2$\\
$a\geq3$, $b=2a$
 & $Z=2Y$ & $(1,0),(0,1),(0,2)$ & $2$\\
$a\geq3$, $b=2a-1$
 & $X+Z=2Y$ & $(1,-1),(1,0),(1,1)$ & $2$\\
$(a,b)=(2,3)$
 & $Y=2X$, $Z=3X$ & $1,2,3$ & $1$\\
$(a,b)=(2,4)$
 & $Y=2X$, $Z=4X$ & $1,2,4$ & $1$\\
\hline\\
\end{tabular}
\caption{Exhaustive short additive relations for three positive steps.
The lifted steps retain the order $1,a,b$.}
\label{tab:three-step-patterns}
\end{table}

\vskip 0.1cm
\emph{Interval lines and their Cartesian products.}
The lifted digraphs in the first, second, fourth, and sixth rows are respectively the following Cartesian products of the interval digraphs $D_d=\Cay(\mathbb Z,\{1,\ldots,d\})$:
\[D_1^{\square3},\qquad D_2\,\square\, D_1,\qquad D_1\,\square\, D_2,\qquad D_3.\]
$D_d$ is acyclic by the contraction in Lemma~\ref{lem:interval-lift}, applied coordinatewise.

\vskip 0.1cm
\emph{The grid with a diagonal step.}
For
\[
D=\Cay(\mathbb Z^2,\{(1,0),(0,1),(1,1)\}),
\]
the coordinate truncations 
$F_k(p,q)=(\max\{p,k\},q)$ carry every finite-support cycle, through graph homotopies, to a
vertical copy of \(D_1\). This treats the third row.

\vskip 0.1cm
\emph{The arithmetic-progression grid.}
Let
\[
D=\Cay(\mathbb Z^2,\{(1,-1),(1,0),(1,1)\}).
\]
For a finite-support cycle, choose \(R\) bounding its second coordinates
and, for \(0\leq t\leq R\), define
\[
c_t(q)=\max\{-R+t,\min\{q,R-t\}\},
\qquad
F_t(p,q)=(p+t,c_t(q)).
\]
The inequalities
\[
|c_t(q+1)-c_t(q)|\leq1,
\qquad
|c_{t+1}(q)-c_t(q)|\leq1
\]
show that the \(F_t\) are graph maps with \(F_t\to F_{t+1}\).
Since \(F_0\) fixes the cycle and \(F_R\) maps it into the horizontal
copy of \(D_1\), the cycle bounds. Thus \(D\) is acyclic, treating the fifth row.

\vskip 0.1cm
\emph{The doubling case \(\{1,2,4\}\).}
For \(D=\Cay(\mathbb Z,\{1,2,4\})\), the map
\[
R(j)=2\left\lceil\frac j2\right\rceil
\]
is a digraph deformation retraction onto the even vertices: it fixes
that subgraph and \(j\to R(j)\) is either an edge or an equality.
The even subgraph is isomorphic to \(D_2\) under \(2j\mapsto j\), and
is therefore acyclic by Lemma~\ref{lem:interval-lift}. Hence \(D\) is acyclic, treating the seventh row.

All hypotheses of Proposition~\ref{prop:acyclic-lift-criterion} now
hold in every row. Applying it proves the formula in all degrees.
\end{proof}

\begin{corollary}
\label{cor:three-step-examples}
In the stable range $n>2\max S$, the complete Betti sequences are
\[
\begin{array}{c|c}
S& (\beta_0,\beta_1,\beta_2,\beta_3,\beta_4,\ldots)\\ \hline
\{1,2,3\},\ \{1,2,4\}&(1,1,0,0,0,\ldots)\\
\{1,2,b\},\ b\geq5&(1,2,1,0,0,\ldots)\\
\{1,a,a+1\},\ a\geq3&(1,2,1,0,0,\ldots)\\
\{1,a,2a-1\},\ a\geq3&(1,2,1,0,0,\ldots)\\
\{1,a,2a\},\ a\geq3&(1,2,1,0,0,\ldots)\\
\{1,4,6\}&(1,3,3,1,0,\ldots).
\end{array}
\]
\end{corollary}

\begin{example}\label{st:146}
For $S=\{1,4,6\}$, the unordered pair sums
$2,5,7,8,10,12$ are distinct and disjoint from $S$. Hence
\[
 (\beta_0,\beta_1,\beta_2,\beta_3)=(1,3,3,1),
 \qquad \beta_m=0\quad(m\ge4)
\]
for every $n\ge13$. When $n$ is invertible in $\K$, the cycle
$Z_{1,4,6}^{(n)}$ in \eqref{st:alternating-cycles} explicitly generates
$H_3$.
\end{example}

\begin{remark}
The collision-free case and the acyclicity of
$\Cay(\mathbb Z,\{1,2,4\})$ are contained in
\cite{DIMZ}. The contractions above
handle the remaining patterns and complete the three-step classification.
For four or more steps, the same method reduces the problem to the path
homology of finitely many short-relation lifts.
\end{remark}

\section{Arbitrarily many steps: boxes and doubling strings}
\label{sec:box-families}
In this section, we apply the general machinery to two families with arbitrarily many steps. 
\begin{itemize}
    \item[$(1)$] Rectangular additive types are governed by acyclic lifts and yield exterior homology over every field. 
    \item[$(2)$] Doubling strings are computed from the dual step-word algebra in characteristic zero, which also gives all boundary ranks.
\end{itemize}

\subsection{Rectangular additive types}

\begin{lemma}
\label{lem:box-lift-acyclic}
Let $d\geq1$, let $k_1,\ldots,k_d\geq1$, and set
\[
 B=\prod_{i=1}^{d}\{0,1,\ldots,k_i\},\qquad
 T_B=\{(1,u):u\in B\}\subset\mathbb Z^{d+1}.
\]
The digraph $D_B=\Cay(\mathbb Z^{d+1},T_B)$ has zero reduced
path homology over arbitrary coefficient field.
\end{lemma}

\begin{proof}
Write its vertices as $(p;y_1,\ldots,y_d)$. Let $z$ be a
finite-support reduced path cycle, and choose integers $A_i\leq B_i$
bounding the $i$th transverse coordinate of every vertex occurring in
$z$. Set $R=\max_i(B_i-A_i)$. For $0\leq r\leq R$, define
\[
 a_{i,r}=\min\{A_i+r,B_i\},\qquad
 c_{i,r}(y)=\max\{y,a_{i,r}\},
\]
and define a map on the whole vertex set by
\[
 F_r(p;y_1,\ldots,y_d)
   =(p+r;c_{1,r}(y_1),\ldots,c_{d,r}(y_d)).
\]
Each $c_{i,r}$ is nondecreasing and $1$-Lipschitz on $\mathbb Z$.
Then, for $0\leq u_i\leq k_i$,
\[
 0\leq c_{i,r}(y_i+u_i)-c_{i,r}(y_i)\leq u_i\leq k_i.
\]
Thus $F_r$ sends every edge of step $(1,u)$ to another allowed edge
of $D_B$. In addition,
\[
 c_{i,r+1}(y_i)-c_{i,r}(y_i)\in\{0,1\}.
\]
Since every $k_i\geq1$, the difference $F_{r+1}(v)-F_r(v)$ belongs
to $T_B$ for every vertex $v$. Hence the maps $F_r$ and $F_{r+1}$
are joined by a one-step digraph homotopy. Their induced chain maps
are chain homotopic.

The map $F_0$ fixes every vertex in $z$, whereas $F_R$ sends those
vertices to the line
\[
 L=\{(p,B_1,\ldots,B_d):p\in\mathbb Z\}.
\]
The induced subgraph on $L$ is $\Cay(\mathbb Z,\{1\})$, which is
acyclic by Lemma~\ref{lem:interval-lift}. Thus $(F_R)_*z$ bounds
in $L$, and the finite sequence of the chain homotopies shows that $z$
bounds in $D_B$. The same argument applies in reduced degree zero.
\end{proof}

\begin{theorem}
\label{thm:box-family}
Let $B$ be as in Lemma~\ref{lem:box-lift-acyclic}, and let
$S=\{s_u:u\in B\}\subset\mathbb Z_{>0}$ contain $1$. Assume that
the indexing map $u\mapsto s_u$ is a bijection satisfying
\begin{align}
 (S+S)\cap S&=\varnothing,
 \label{eq:box-no-merge}\\
 s_u+s_v=s_w+s_x
 &\quad\Longleftrightarrow\quad u+v=w+x
 \qquad(u,v,w,x\in B).
 \label{eq:box-pair-type}
\end{align}
Then $Q_S\cong\mathbb Z^{d+1}$, with the lifted step associated to
$s_u$ equal to $(1,u)$. For every field $\K$ and every
$n>2\max S$ (or, more generally, every admissible $n$),
\[
 H_m^{\mathrm{path}}(\vec C_n^S;\K)
       \cong\Lambda^m(\K^{d+1})\qquad(m\geq0),
 \qquad
 P_{\vec C_n^S}(z)=(1+z)^{d+1}.
\]
For the infinite cyclic lift $X_S=\Cay(\mathbb Z,S)$, one has the isomorphisms of graded vector spaces
\[
 H_m^{\mathrm{path}}(X_S;\K)
       \cong\Lambda^m(\K^d),\qquad P_{X_S}(z)=(1+z)^d.
\]
\end{theorem}

\begin{proof}
By \eqref{eq:box-no-merge}, the group $Q_S$ has only pair-sum
relations. Write $E_u$ for the generator associated to $s_u$ and
let $e_i$ denote the $i$th coordinate vector in $B$. Whenever
$u_i>0$, the relation
\[
 E_u+E_0=E_{u-e_i}+E_{e_i}
\]
holds. Induction on $u_1+\cdots+u_d$ gives
\[
 E_u=E_0+\sum_{i=1}^d u_i(E_{e_i}-E_0).
\]
These $d+1$ generators have no further relations: the homomorphism
defined by $E_u\mapsto(1,u)$ is well defined by
\eqref{eq:box-pair-type}, and it sends $E_0,E_{e_i}-E_0$ to the
standard basis of $\mathbb Z^{d+1}$. 

Lemma~\ref{lem:box-lift-acyclic} and
Proposition~\ref{prop:acyclic-lift-criterion} prove the finite
quotient formula. To compute $X_S$, use instead the epimorphism
\[
 \ell:Q_S\longrightarrow\mathbb Z,\qquad E_u\longmapsto s_u.
\]
It is surjective because $1\in S$, and its kernel is free abelian of rank $d$. The equality of short relations gives the same free-resolution and coinvariant argument as in that proposition, with deck group $\ker\ell$ in place of $N_n$. Hence the homology of $X_S$ is the group homology of $\ker\ell$, proving the second formula.
\end{proof}

\begin{example}
\label{ex:arithmetic-progressions}
For $q\geq2$ and $k\geq1$, set
\[
 S_{q,k}=\{1,1+q,\ldots,1+kq\}.
\]
No pair sum belongs to \(S_{q,k}\), and
\[
(1+iq)+(1+jq)=(1+uq)+(1+vq)
\quad\Longleftrightarrow\quad
i+j=u+v.
\]
Thus for every field $\K$ and $n>2(1+kq)$, Theorem~\ref{thm:box-family} with \(d=1\) gives
\[
 H_m^{\mathrm{path}}(\vec C_n^{S_{q,k}};\K)
       \cong\Lambda^m(\K^2).
\]
Thus its full Betti sequence is $(1,2,1,0,\ldots)$, independently
of the number $k+1$ of steps. Its infinite cyclic lift has Betti
sequence $(1,1,0,\ldots)$.
\end{example}

\begin{example}
For \(r\geq1\), set
\[
S_r=
\left\{
1+2\sum_{i=1}^{r}3^{i-1}u_i:
u_i\in\{0,1\}
\right\}.
\]
Every element of \(S_r\) is odd, so no pair sum belongs to \(S_r\).
Moreover, uniqueness of base-\(3\) expansions with digits
\(0,1,2\) gives
\[
s_u+s_v=s_w+s_x
\quad\Longleftrightarrow\quad
u+v=w+x.
\]
Thus Theorem~\ref{thm:box-family} applies and, for every field \(\K\)
and \(n>2\cdot3^r\),
\[
H_m^{\mathrm{path}}(\vec C_n^{S_r};\K)
   \cong\Lambda^m(\K^{r+1}),
\qquad
P_{\vec C_n^{S_r}}(z)=(1+z)^{r+1}.
\]
Since \(|S_r|=2^r\), this gives explicit families with arbitrarily
many steps and nonzero path homology in arbitrarily high degrees.
\end{example}

\subsection{Doubling strings and boundary ranks}
\label{subsec:doubling-strings}

We now turn to step sets whose only legal merges are doublings. Throughout
this subsection $\operatorname{char}\K=0$. The calculation uses the graded
dual of the stable step-word complex.

Let \(C_*(S)\) be the stable step-word complex, and let \(x_s\) be
dual to the one-letter word \(s\). The local relations defining
\(C_*(S)\) give the graded-dual presentation
\begin{equation}\label{eq:examples-dual-algebra}
\mathcal A(S)=
T(x_s:s\in S)\Big/
\left(
\sum_{a+b=r}x_ax_b:
r\in(S+S)\setminus S
\right).
\end{equation}
Deconcatenation restricts to \(C_*(S)\), so its dual product on
\(\mathcal A(S)\) is concatenation. The dual differential is the
degree-one derivation determined by
\begin{equation}\label{eq:examples-dual-differential}
\delta x_s
 =\sum_{a\in S}(x_ax_s+x_sx_a)
  -\sum_{\substack{a,b\in S\\a+b=s}}x_ax_b.
\end{equation}
Thus \((\mathcal A(S),\delta)\) is the differential graded algebra
dual to \((C_*(S),\delta)\).

\begin{theorem}\label{thm:doubling-strings}
Let $S\subset\mathbb Z_{>0}$ be finite with $1\in S$ and let
$\operatorname{char}\K=0$. Assume that:
\begin{enumerate}
\item If $a+b=c+d$ with $a,b,c,d\in S$, then the unordered pairs
      $\{a,b\}$ and $\{c,d\}$ coincide, with multiplicities.
\item If $a+b\in S$ for $a,b\in S$, then $a=b$.
\end{enumerate}
Let $D=\{s\in S:2s\in S\}$ and let $b=|S|-|D|$.
For every $n>2\max S$,
\[
\dim_\K H_m^{\mathrm{path}}(\vec C_n^S;\K)=\binom{b}{m}
\quad(m\ge0).
\]
In particular, the dimensions vanish for $m>b$.
\end{theorem}

\begin{proof}
The directed graph on $S$ with arrows $s\rightarrow2s$ is a
disjoint union of finite strings: every vertex has at most one
predecessor and one successor, and positivity excludes cycles.
Write its components as
\[
S_i=\{r_i,2r_i,\ldots,2^{\ell_i}r_i\},\qquad 1\le i\le b.
\]
Indeed, its $|D|$ edges and $|S|$ vertices give exactly $b$
components. The hypotheses and
\eqref{eq:examples-dual-algebra} give the presentation
\[
\mathcal A(S)=\K\langle x_s:s\in S\rangle/
\bigl(x_sx_t+x_tx_s\ (s\ne t),\quad x_s^2\ (s\notin D)\bigr).
\]
The ordered monomials form a basis: distinct generators can be
sorted using their sign commutation rules, the exponents of terminal
generators are at most one, and the other exponents are arbitrary.
In particular, if $d=|D|$,
\begin{equation}\label{eq:doubling-hilbert}
\sum_{m\ge0}\dim C_m(S)t^m=\frac{(1+t)^b}{(1-t)^d}.
\end{equation}

For $s\in D$, set $u_s=x_s^2$. These elements are central,
algebraically independent, of degree two, and closed. Thus
$\mathcal A(S)$ is free over $R=\K[u_s:s\in D]$, with basis the
ordered products in which each $x_s$ has exponent zero or one.
Writing $x_{i,j}=x_{2^j r_i}$ and $u_{i,j}=x_{i,j}^2$ for
$0\le j<\ell_i$, the differential is
\[
\delta x_{i,j}=2u_{i,j}-u_{i,j-1},
\qquad u_{i,-1}=u_{i,\ell_i}=0.
\]

For each string, define
\[
p_{i,j}=\sum_{h=0}^{j}2^{h-j-1}x_{i,h}
\quad(0\le j<\ell_i),
\qquad
q_i=\sum_{h=0}^{\ell_i}2^h x_{i,h}.
\]
Then we have
\[
\delta p_{i,j}=u_{i,j},\qquad\delta q_i=0.
\]
The change of variables from the \(x_{i,j}\) to the \(p_{i,j},q_i\)
is triangular with nonzero diagonal. Hence the ordered square-free
monomials in the \(p_{i,j},q_i\) form an \(R\)-basis of
\(\mathcal A(S)\).

Thus, as a complex of graded vector spaces,
\[
\mathcal A(S)\cong
\bigotimes_{s\in D}
\bigl(\K[u_s]\oplus\K[u_s]p_s,\ \delta p_s=u_s\bigr)
\otimes_\K E,
\]
where \(E\) has zero differential and basis
\[
q_{i_1}\cdots q_{i_m},
\qquad i_1<\cdots<i_m.
\]
Each elementary factor has cohomology \(\K\) concentrated in degree
zero. Therefore
\[
\dim_\K H^m(\mathcal A(S),\delta)=\binom bm.
\]
Degreewise duality and Theorem~\ref{cl:stability} now give the claimed path Betti numbers of $\vec{C}_n^S$.
\end{proof}

\begin{remark}
\label{rem:doubling-generators}
For \(i_1<\cdots<i_m\), the alternating step-word chain
\[
\sum_{\sigma\in\mathfrak S_m}\operatorname{sgn}(\sigma)
[r_{i_{\sigma(1)}},\ldots,r_{i_{\sigma(m)}}]
\]
is a cycle. Summing it over all initial vertices gives a basis of path
homology. Let
\[
a_m=\operatorname{rank}
\bigl(\delta:C_m(S)\to C_{m-1}(S)\bigr).
\]
Then \(a_0=a_1=0\), and \eqref{eq:doubling-hilbert} gives
\[
\sum_{m\geq1}a_mt^m
=t(1+t)^{b-1}\bigl((1-t)^{-d}-1\bigr).
\]
On each nontrivial Fourier block, the corresponding rank series is $t(1+t)^{b-1}(1-t)^{-d}$.
\end{remark}

\begin{example}
Let us re-compute two three-step examples in Corollary \ref{cor:three-step-examples}: $S=\{1,2,4\}$ and $S=\{1,2,5\}$. Write $d=|D|$.
\begin{itemize}
    \item For $S=\{1,2,4\}$, one has $(b,d)=(1,2)$ and
$\dim C_m(S)=2m+1$ for $m\ge1$. The invariant boundary has
rank $m$ for $m\ge2$; each nontrivial Fourier block has rank $m$
in degree $m\ge1$. Thus the Betti sequence is $(1,1,0,\ldots)$.
    \item For $S=\{1,2,5\}$, one has $(b,d)=(2,1)$,
$\dim C_1(S)=3$, and $\dim C_m(S)=4$ for $m\ge2$.
The invariant ranks are $a_1=0$, $a_2=1$, and $a_m=2$ for $m\ge3$;
the nontrivial ranks are $1$ in degree one and $2$ in all degrees
$m\ge2$. Thus the Betti sequence is $(1,2,1,0,\ldots)$.
These chain dimensions and ranks follow from
\eqref{eq:doubling-hilbert} and Remark~\ref{rem:doubling-generators}.
\end{itemize}
\end{example}

\begin{example}
For $S=\{1,2,4,\ldots,2^k\}$ there is one doubling string, so
$\beta_0=\beta_1=1$ and all higher Betti numbers vanish for
$n>2^{k+1}$.    
\end{example}

\appendix

\section{Marked lattice quotients and symmetric circulants}\label{sec:circulantgraph}

This appendix records the marked lattice-quotient presentation of
circulants \cite{Louis} and compares a directed example with its
symmetrization.

\subsection{The lattice quotient and its marked generators}
\label{sub:discretetori}
Let $B\in M_d(\mathbb Z)$ have nonzero determinant. Its Smith normal form gives unimodular matrices $U,V\in\mathrm{GL}_d(\mathbb Z)$ such that
\[
UBV=\operatorname{diag}(t_1,\ldots,t_d),\qquad t_i\mid t_{i+1},
\qquad \mathbb Z^d/B\mathbb Z^d\cong\prod_{i=1}^d\mathbb Z_{t_i}.
\]
This identifies the quotient group. A Cayley graph additionally remembers the images of its chosen generators; the left basis change $U$ transforms those generators and need not preserve the standard nearest-neighbor connection set.

For $S=\{1,\gamma_1,\ldots,\gamma_{d-1}\}$, set
\begin{equation}\label{eq:LambdaS}
\Lambda_S=\begin{pmatrix}
n&-\gamma_1&\cdots&-\gamma_{d-1}\\
0&&I_{d-1}&
\end{pmatrix}.
\end{equation}
The homomorphism
\[
\Phi:\mathbb Z^d\longrightarrow\mathbb Z_n,
\qquad (v_1,\ldots,v_d)\longmapsto v_1+\sum_{j=2}^d\gamma_{j-1}v_j
\]
has kernel $\Lambda_S\mathbb Z^d$: the columns lie in its kernel, its image contains $1$, and both indices are $n$. Consequently
\begin{equation}\label{eq:cyclic}
\mathbb Z^d/\Lambda_S\mathbb Z^d\cong\mathbb Z_n.
\end{equation}
Its ordered Smith normal form is $\operatorname{diag}(1,\ldots,1,n)$. The images of the standard basis vectors are nevertheless the distinct steps $1,\gamma_1,\ldots,\gamma_{d-1}$, and these must be retained.

\begin{proposition}
Equip $\mathbb Z^d/\Lambda_S\mathbb Z^d$ with the images of $\pm e_1,\ldots,\pm e_d$ as its connection set. Then the map induced by $\Phi$ is a graph isomorphism onto $C_n^S$. With only the positive $e_j$ it is a digraph isomorphism onto $\vec C_n^S$.
\end{proposition}
\begin{proof}
The induced group map is bijective and sends the stated generators to exactly $\pm S$, respectively $S$, so it preserves and reflects edges.
\end{proof}
\subsection{Path complex of circulant graphs}\label{sub:circulatgraph}

We regard the undirected circulant as the symmetric digraph
$C_n^S=\Cay(\Z_n,S\cup(-S))$. Adding the reverse steps can create
new fillings and change its path homology.

\begin{example}
For $n=5$, $S=\{1,2\}$, 
the path homologies of the circulant digraph $\vec{C}_5^{1,2}$ and the circulant graph $C_5^{1,2}$ (now $C_5^{1,2}$ is a complete graph) are given by
\[
H_m^{\mathrm{path}}(\vec{C}_5^{1,2};\K)\cong\begin{cases}
    \K  & m=0,1\\ 
    0  & m\geq2
\end{cases}~,\qquad
H_m^{\mathrm{path}}(C_5^{1,2};\K)\cong\begin{cases}
    \K  & m=0\\ 
    0  & m\geq1
\end{cases}~.
\]    
\end{example}

\smallskip
\begin{example}

For $n=7$, $S=\{1,3\}$, see Fig. \ref{fig:C713}.
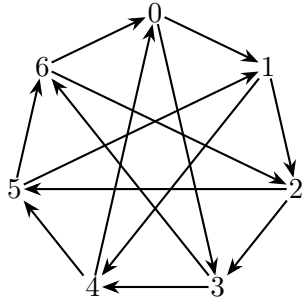
\begin{figure}[htbp]
\centering
\begin{tikzpicture}[
  scale=0.8,
  >=Stealth,
  node/.style={inner sep=0.6pt}
]
\def\n{7}
\def\r{2.4}

\foreach \i in {0,...,6} {
  \node[node] (\i) at ({90-360*\i/\n}:\r) {\i};
}

\foreach \i in {0,...,6} {
  \foreach \s in {1,3} {
    \pgfmathtruncatemacro{\j}{mod(\i+\s,\n)}
    \draw[->,line width=0.8pt] (\i) -- (\j);
  }
}

\end{tikzpicture}
\caption{circulant digraph $\vec{C}_7^{1,3}$}
\label{fig:C713}
\end{figure}

we compute the path homologies of the circulant digraph $\vec{C}_7^{1,3}$ in Example 4.3 of~\cite{TY2}: 
\begin{align*}
&H_0^{\mathrm{path}}(\vec{C}_7^{1,3};\K)\cong\K;\\
&H_1^{\mathrm{path}}(\vec{C}_7^{1,3};\K)=\Span_{\K}\bigg\{[e_{03}+e_{36}+e_{60}],[e_{01}+e_{12}+e_{23}+e_{34}+e_{45}+e_{56}+e_{60}]\bigg\};\\
&H_2^{\mathrm{path}}(\vec{C}_7^{1,3};\K)=\Span_{\K}\bigg\{[(e_{014}-e_{034})+(e_{125}-e_{145})+(e_{236}-e_{256})+(e_{340}-e_{360})\\
&\qquad\qquad\qquad\qquad\qquad\quad+(e_{451}-e_{401})+(e_{562}-e_{512})+(e_{603}-e_{623})]\bigg\};\\
&H_m^{\mathrm{path}}(\vec{C}_7^{1,3};\K)=0,\quad m\geq 3.
\end{align*}

\begin{figure}[h]
\centering
\begin{tikzpicture}[
  scale=0.8,
  >=Stealth,
  node/.style={inner sep=0.6pt}
]
\def\n{7}
\def\r{2.4}

\foreach \i in {0,...,6} {
  \node[node] (\i) at ({90-360*\i/\n}:\r) {\i};
}

\foreach \i in {0,...,6} {
  \pgfmathtruncatemacro{\j}{mod(\i+1,\n)}
  \draw[<->,line width=0.8pt] (\i) -- (\j);

  \pgfmathtruncatemacro{\k}{mod(\i+3,\n)}
  \draw[<->,line width=0.8pt] (\i) -- (\k);
}

\end{tikzpicture}
\caption{circulant graph $C_7^{1,3}$ as a symmetric digraph}
\label{fig:C713pm}
\end{figure}
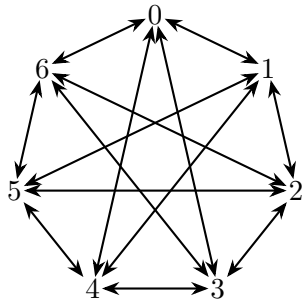

\noindent But for the symmetric digraph $C_7^{1,3}$, see Fig. \ref{fig:C713pm}, we have  
\[S\cup (-S)=\{1,3,-1,-3\}=\{1,3,4,6\}\subset \Z_7.\]
Consider the $3$-paths
\[P^{1,3}=\sum_{a\in\Z_7}e_{a,a+1,a+4,a+5},\qquad P^{4,4}=\sum_{a\in\Z_7}e_{a,a+4,a+1,a+5},\]
by Definition \ref{def:Wcycle} and Proposition \ref{prop:P}, we have
\[\partial P^{1,3}=W^{1,4}-W^{1,3},\qquad \partial P^{4,4}=W^{4,1}-W^{4,4}=-W^{1,4}.\]
Then $W^{1,3}$ becomes a boundary
\[W^{1,3}=\partial(-P^{1,3}-P^{4,4}).\]

In fact, the symmetric circulant $C_7^{1,3}$ is path-homotopy contractible, and then has trivial path homology. Define $f:\mathbb Z_7\to\mathbb Z_7$ by
\[
f(0)=f(1)=f(3)=f(4)=f(6)=0,
\qquad f(2)=f(5)=1.
\]
Its image is the bidirected edge
$0\leftrightarrow1$, so $f$ is a digraph map. For every $x$, either $f(x)=x$ or $x$ and $f(x)$ are joined
by a bidirected edge. Hence $\operatorname{id}\simeq f$. If
$c_0$ denotes the constant map with value $0$, then
$f\simeq c_0$, since $f(x)\in\{0,1\}$ and $0\leftrightarrow1$.
\end{example}

\section{Cohomology rings and explicit generators}
\label{app:cohomology}

This appendix records the multiplicative refinement of the
short-relation-lift calculation. We use the path-cochain product of
\cite{GLMY:cohomology} and the standard comparison of projective
resolutions \cite[Chapter~V]{Brown1982}.

\begin{lemma}\label{lem:path-AW-diagonal}
For every digraph $G$, the Alexander--Whitney formula
\[
 \Delta e_{v_0\cdots v_m}
 =\sum_{p=0}^{m}e_{v_0\cdots v_p}\otimes e_{v_p\cdots v_m}
\]
restricts to a coassociative, counital chain map
\[
 \Delta:\Omega_*(G;\K)\longrightarrow
 \Omega_*(G;\K)\otimes_\K\Omega_*(G;\K).
\]
It is natural under digraph maps and equivariant under graph
automorphisms. Consequently the path cochain complex
\[
C^*(G;\K)=\operatorname{Hom}_\K(\Omega_*(G;\K),\K)
\]
is a differential graded algebra with product
\[
 (\alpha\smile\beta)(u)
  =(\alpha\otimes\beta)\Delta_{p,q}u,
 \qquad \alpha\in C^p(G;\K),\quad\beta\in C^q(G;\K).
\]
\end{lemma}

\begin{proof}
On regular paths, the Alexander--Whitney map is a coassociative, counital chain map. Since every prefix and suffix of an allowed path is allowed, its \((p,q)\)-component satisfies
\[
\Delta_{p,q}(\mathcal A_{p+q})\subseteq \mathcal A_p\otimes\mathcal A_q.
\]

Let
\[
b_j:\mathcal A_j\longrightarrow
\mathcal R_{j-1}/\mathcal A_{j-1},
\qquad
b_j(u)=\partial u \bmod \mathcal A_{j-1},
\]
so that \(\Omega_j=\ker b_j\). If \(u\in\Omega_{p+q}\), then
\(\partial u\) is allowed. Projecting the Alexander--Whitney chain-map
identity
\[
\Delta\partial
 =(\partial\otimes1+(-1)^{\deg}1\otimes\partial)\Delta
\]
onto the forbidden quotient in the first and second factors gives
\[
(b_p\otimes1)\Delta_{p,q}u=0,
\qquad
(1\otimes b_q)\Delta_{p,q}u=0.
\]
Since \(\K\) is a field,
\[
\ker(b_p\otimes1)\cap\ker(1\otimes b_q)
   =\Omega_p\otimes\Omega_q.
\]
Thus \(\Delta\) restricts to \(\Omega_*(G;\K)\).
Naturality, coassociativity, and the counit follow directly from the
vertex formula. Dualizing the chain-map and coassociativity identities
gives the differential graded algebra structure on \(C^*(G;\K)\).
\end{proof}

\begin{theorem}
\label{thm:multiplicative-cover}
Under the hypotheses of
Proposition~\ref{prop:acyclic-lift-criterion}, there is an
isomorphism of unital graded $\K$-algebras
\[
 H^*_{\mathrm{path}}(\vec C_n^S;\K)
   \cong H^*(N_n;\K)
   \cong \Lambda^*_\K\operatorname{Hom}_{\mathbb Z}(N_n,\K).
\]
Here $\K$ in $\operatorname{Hom}_{\mathbb Z}$ denotes its additive
group, so the degree-one space has dimension $r$. The degree-one generators have square zero also in characteristic two.
\end{theorem}

\begin{proof}
Set $N=N_n$, $A=\K[N]$, and
$P_*=\Omega_*(\widetilde G_S;\K)$.
The proof of Proposition~\ref{prop:acyclic-lift-criterion} shows
that $P_*\to\K$ is a free $A$-resolution and that
\[
 \Omega_*(\vec C_n^S;\K)=\K\otimes_A P_*.
\]
Thus the path cochains identify with
$\operatorname{Hom}_A(P_*,\K)$, where $N$ acts trivially on
$\K$. Lemma~\ref{lem:path-AW-diagonal} equips $P_*$ with an
$A$-equivariant diagonal into $P_*\otimes_\K P_*$, where $N$
acts diagonally. Projection to coinvariants in the two factors
recovers exactly the diagonal on the quotient graph. Thus the cochain identification preserves the path cup product.

Choose a basis $g_1,\ldots,g_r$ of the free abelian group $N$.
Its Koszul resolution is
\[
 F_*=A\otimes_\K\Lambda_\K(e_1,\ldots,e_r),
 \qquad \partial e_i=g_i-1,
\]
with $\partial$ extended as a graded derivation. It has an
$A$-equivariant diagonal defined by
\[
 \Delta_F(g_i)=g_i\otimes g_i,
 \qquad
 \Delta_F(e_i)=e_i\otimes1+g_i\otimes e_i,
\]
and extended multiplicatively in the graded tensor product.
These formulas respect the exterior relations, even in
characteristic two, and satisfy
$\partial\Delta_F=\Delta_F\partial$.

Let \(f:P_*\to F_*\) be a comparison map lifting
\(\operatorname{Id}_{\K}\). Since \(F_*\otimes_{\K}F_*\), with the diagonal \(N\)-action, is a free \(A\)-resolution of \(\K\), the comparison theorem gives an \(A\)-linear chain homotopy
\[
(f\otimes f)\Delta_P\simeq \Delta_F f.
\]
Hence \(f^*\) preserves cup products on cohomology.

Since \(N\) acts trivially on \(\K\), the differential on
\(\operatorname{Hom}_A(F_*,\K)\) vanishes, and the Koszul diagonal
induces the standard exterior product on its degree-one generators.
Therefore
\[
H^*(N;\K)
   \cong \Lambda_{\K}^*\operatorname{Hom}_{\mathbb Z}(N,\K).
\]
Together with the multiplicative isomorphism induced by \(f^*\), this
proves the theorem.
\end{proof}

The following formula gives explicit degree-one representatives over
any field. Choose a section
$\sigma:\mathbb Z_n\to Q_S$ of $\phi_n$, with $\sigma(0)=0$.
For $\lambda\in\operatorname{Hom}_{\mathbb Z}(N_n,\K)$ set
\begin{equation}\label{eq:deck-cocycle}
 \alpha_\lambda(x,x+s)
   =\lambda\bigl(\sigma(x)+\bar E_s-\sigma(x+s)\bigr).
\end{equation}
The expression in parentheses belongs to $N_n$.
On the lift, this cochain is $dh$ for
\[
 h(q)=\lambda\bigl(q-\sigma(\phi_n(q))\bigr),
\]
so it is closed after descending to the quotient. Replacing
$\sigma$ by another section changes
\eqref{eq:deck-cocycle} by a coboundary. Its period on a loop
whose lift has deck displacement $u\in N_n$ is $\lambda(u)$,
which identifies its class with $\lambda\in H^1(N_n;\K)$.
Its period on a loop whose lift has deck displacement \(u\in N_n\)
is \(\lambda(u)\), which identifies its class with
\(\lambda\in H^1(N_n;\K)\). A basis of the dual lattice therefore
gives explicit generators
\[
[\alpha_1],\ldots,[\alpha_r],
\]
whose ordered cup products form a basis in every degree.

\begin{corollary}
\label{cor:classified-exterior-rings}
Every three-step graph in Theorem~\ref{thm:complete-three-step} has path
cohomology ring exterior on $r(a,b)$ degree-one generators. For the box
family in Theorem~\ref{thm:box-family}, the path cohomology ring is
exterior on $d+1$ degree-one generators.
\end{corollary}

\begin{proof}
Apply Theorem~\ref{thm:multiplicative-cover} to the acyclic lifts
constructed in the proofs of Theorems~\ref{thm:complete-three-step}
and~\ref{thm:box-family}.
\end{proof}

\end{document}